\documentclass[10pt]{article}

\usepackage{vmargin}

\setpapersize{A4}
\setmargins{3.8cm}       % margen izquierdo
{1.5cm}                        % margen superior
{13.5cm}                      % anchura del texto
{23.42cm}                    % altura del texto
{10pt}                           % altura de los encabezados
{1cm}                           % espacio entre el texto y los encabezados
{0pt}                             % altura del pie de página
{2cm}

\usepackage{hyperref}
\usepackage{amsfonts}
\usepackage{amsmath}
\usepackage{amsthm}
\usepackage{amssymb}
\usepackage[all]{xy}
\usepackage[latin1]{inputenc}
\usepackage{graphicx}
\usepackage{color}
\usepackage{yfonts}
\input xy
\xyoption{all}
\usepackage{graphicx}
\usepackage{mathrsfs}
\usepackage{float}
\usepackage{framed}
\usepackage{titlesec}

\newtheorem{theorem}{Theorem}[section]
\newtheorem{corollary}[theorem]{Corollary}

\newtheorem{proposition}[theorem]{Proposition}
\theoremstyle{definition}
\newtheorem{definition}[theorem]{Definition}
\newtheorem{remark}[theorem]{Remark}

\newtheorem*{notation}{Notation}

\newcommand{\OO}{\mathscr{O}}
\newcommand{\Q}{\mathcal{Q}^{0}}
\newcommand{\q}{\mathfrak{q}}
\newcommand{\slas}{/\!\!/ }
\newcommand{\GG}{\mathbb{G}_{m}}

\newcommand{\G}{\mathscr{G}}

\newcommand{\Pp}{\mathbb P}

\newcommand{\N}{\mathscr{N}}

\newcommand{\Lcc}{\mathscr L}

\newcommand{\K}{\mathscr{K}}

\newcommand{\Hh}{\mathscr{M}}

\titleformat{\section}[display]{\scshape \bfseries}{\S \thesection\filcenter}{1ex}{\fillast}

\titleformat{\subsection}[hang]{\itshape\bfseries}{\thesubsection. --- }{0pt}{\upshape}

\titleformat{\subsubsection}[runin]{\itshape\filleft}{\thesubsubsection.---}{0pt}{}

\newcommand{\eq}[1][r]
   {\ar@<-3pt>@{-}[#1]
    \ar@<-1pt>@{}[#1]|<{}="gauche"
    \ar@<+0pt>@{}[#1]|-{}="milieu"
    \ar@<+1pt>@{}[#1]|>{}="droite"
    \ar@/^2pt/@{-}"gauche";"milieu"
    \ar@/_2pt/@{-}"milieu";"droite"}

\date{}
\title{
\textbf{Generalized parabolic structures over smooth curves with many components and principal bundles over reducible nodal curves}
\author{
{\footnotesize \'ANGEL LUIS MU\~NOZ CASTA\~NEDA} 
\footnote{Department of Mathematics, University of Le\'on, Spain, email: {\texttt amunc@unileon.es}}
}
}

\begin{document}

\maketitle

\begin{abstract}
Let $Y_1,\hdots,Y_l$ be smooth irreducible projective curves and let $Y$ be its disjoint union. Given a semisimple linear algebraic group $G$ and a faithful representation $\rho:G\hookrightarrow \textrm{SL}(V)$ we construct a projective moduli space of $(\underline{\kappa},\delta)$-(semi)stable singular principal $G$-bundles with  generalized parabolic structure of type $\underline{e}$. In case $Y$ is the normalization of a connected and reducible projective nodal curve $X$, there is a closed subscheme 
coarsely representing the subfunctor corresponding to descending bundles. We prove that the descent operation induces a birational, surjective and proper morphism onto the schematic closure of
the space of $\delta$-stable singular principal $G$-bundles whose associated torsion free sheaf is of local type $\underline{e}$.
\end{abstract}

\vspace{1cm}



\tableofcontents

\section{Introduction}
%\fancyhead[RO,LE]{}

Let $X$ be a smooth projective curve over the field of complex numbers $\mathbb{C}$, $\mathscr{E}$ a locally free sheaf on $X$ and $p\in X$ a closed point. A parabolic structure on $\mathscr{E}$ at $p$ is just a flag of vector spaces $(0)\subset E_{1}\subset\cdots\subset E_{s}\subset \mathscr{E}_{p}/\mathfrak{m}_{p}\mathscr{E}_{p}$ together with weights $0\leq \kappa_{1}<\kappa_{1}<\cdots<\kappa_{s}<1$ (weighted flags for short). The study of  parabolic locally free sheaves began with the seminal work of V. B. Mehta and C. S. Seshadri \cite{ms}. They defined a (natural) (semi)stability condition for such objects and proved the existence of a coarse projective moduli space for (semi)stable parabolic locally free sheaves. Furthermore, they proved that the isomorphism classes of parabolic locally free sheaves that are stable coincides with the set of equivalence classes of irreducible unitary representations of the topological fundamental group of $X$ (see \cite[Theorem 4.1]{ms}).

The concept of parabolic locally free sheaf can be generalized by considering weighted flags supported on divisors of the smooth projective curve $X$. These objects are called generalized parabolic locally free sheaves and they where introduced by U. Bohsle in \cite{genparbund}. The importance of generalized parabolic locally free sheaves is not only the possible link to the space of representations of the topological fundamental groups but also the link to the geometry of the moduli spaces of torsion free sheaves on nodal curves. To be more precise, U. Bohsle proved that if $\pi:Y\rightarrow X$ is the normalization map of a reducible projective nodal curve then there exists a coarse projective moduli space for generalized parabolic locally free sheaves (the parabolic structure being supported on $q_1+q_2=\pi^{-1}(p)$) on $Y$ together with a morphism to the moduli space of torsion free sheaves on $X$  of rank $r$ and degree $d$ making the former moduli space a desingularization of the later provided $(r,d)=1$ (see \cite{bosle-multiple}).

Likewise, generalized parabolic structures have been applied for studying the geometry of the moduli space of Hitchin pairs over a reducible curve. In  \cite{hitchinpairs}, U. Bhosle constructs a morphism between the moduli space of Hitchin pairs with generalized parabolic structure over the normalization  $Y$ and the moduli space of Hitchin pairs over the reduced curve $X$, showing that under certain condition this is a birrational morphism whose image contains all stable Higgs bundles.

These ideas have also been applied to the more general problem of studying the compactification of the moduli space of principal $G$-bundles over an irreducible nodal curve. In \cite{Alexander2}, A. Schmitt 
realized that, once a faithful representation $\rho:G\hookrightarrow \textrm{SL}(V)$ is fixed, every principal  $G$-bundle can be seen as a pair $(\mathscr{E},\tau)$ formed by a locally free sheaf $\mathscr{E}$ and a non-trivial morphism of algebras $\tau:S^{\bullet}(V\otimes\mathscr{E})^{G}\rightarrow \mathscr{O}_{X}$. These objects are called singular principal $G$-bundles and they carry a semistability condition, which depends (a priori) on a positive rational parameter $\delta\in\mathbb{Q}_{>0}$. Then, the main result is that there exists a coarse projective moduli space for $\delta$-(semi)stable singular principal $G$-bundles and it coincides with the classical moduli space provided $\delta$ is large enough. This motivated the works \cite{bosle,Alexander4,Alexander1}, where U. Bohsle generalized the definition of singular principal $G$-bundles, as well as the $\delta$-(semi)stability condition, over an irreducible nodal curve in a natural way and proved the existence of a projective moduli space for them, while A. Schmitt studied the asymptotic behavior of the $\delta$-(semi)stability condition obtaining a similar result as that of the smooth case. The study of the asymptotic behavior of the $\delta$-(semi)stability condition becomes harder when the curve has singularities, and it was carried out in \cite{Alexander4,Alexander1} by considering  singular principal $G$-bundles on $X$ as singular principal $G$-bundles with generalized parabolic structures on the normalization $Y$. Therefore, the moduli spaces of  singular principal $G$-bundles with generalized parabolic structures over a smooth projective curve play an important role in this problem.

On the other hand, singular principal $G$-bundles with generalized parabolic structures have been applied to the construction of a compactification of the moduli space of principal Higgs $G$-bundles over an irreducible nodal curve (see \cite{higgs} for instance). In this case, A. Lo Giudice and A. Pustetto enlarge the category of principal Higgs $G$-bundles on the nodal curve to the category of singular principal $G$-bundles together with a Higgs field, which can be seen as singular principal $G$-bundles with generalized parabolic structure together with a Higgs field on the normalization of the nodal curve. Again, the moduli space of the last objects plays an important rol in the study of the moduli space of the first objects.

\subsection*{Goal of the paper}
Let $X$ be a projective nodal curve with nodes $x_{1},\hdots,x_{\nu}$ and $l$ irreducible components, and $\pi\colon Y=\coprod_{i=1}^{l} Y_{i}\rightarrow X$ its normalization. We fix an ample invertible sheaf $\mathscr{O}_{X}(1)$ on $X$ and we denote by $\mathscr{O}_{Y}(1)$ the ample invertible sheaf obtained by pulling $\mathscr{O}_{X}(1)$ back to $Y$. We denote by $h$ the degree of $\mathscr{O}_{Y}(1)$, by $y_{1}^{i},y_{2}^{i}$ the points in the preimage of the $i$th nodal point $x_{i}$, by $D_{i}=y_{1}^{i}+y_{2}^{i}$ the corresponding divisor on $Y$ and by $D=\sum D_{i}$ the total divisor. Let $G$ be a semisimple linear algebraic group, $\rho:G\hookrightarrow \textrm{SL}(V)$ a faithful representation of dimension $r\in\mathbb{N}$, $\delta\in\mathbb{Q}_{>0}$ and $d\in\mathbb{Z}$. Let $\textrm{SPB}(\rho)^{\delta-(s)s}_{r,d}$ be the moduli space of $\delta$-(semi)stable singular principal $G$-bundles of rank $r$ and degree $d$ over $X$ (see \cite{AMC}). Consider the set $J(r)=\{(e_1,\hdots,e_{\nu})\in\mathbb{N}^{\nu}| \ 1\leq e_{i}\leq r\}$. Then, there is a stratification, $\textrm{SPB}(\rho)^{\delta-(s)s}_{r,d}:=\bigcup_{\underline{e}\in J(r)}\textrm{SPB}(\rho)_{r,d,\underline{e}}^{\delta\textrm{-(s)s}}$, where $\textrm{SPB}(\rho)_{r,d,\underline{e}}^{\delta\textrm{-(s)s}}$ parametrizes singular principal bundles, $(\mathscr{F},\tau)$, with $\mathscr{F}_{x_{i}}\simeq \mathscr{O}_{X,x_{i}}^{e_{i}}\oplus \mathfrak{m}_{x_{i}}^{r-e_{i}}$.
The goal of this paper is to construct a coarse projective moduli space, $\textrm{D}(\rho)_{r,d(\underline{e},r),\underline{e}}^{(\underline{\kappa},\delta)\textrm{-(s)s}}$, for $(\underline{\kappa},\delta)$-(semi)stable descending singular principal $G$-bundles with generalized parabolic structures over $Y=\coprod_{i=1}^{l} Y_{i}$ of given type $\underline{e}$ supported on the divisors $D_{i}$ (see Theorem \ref{D}) together with a morphism (see Equation \ref{finalfinal})
\begin{equation*}
\Theta: \textrm{D}(\rho)_{r,d}^{(\underline{\kappa},\delta)\textrm{-(s)s}}:=  \coprod_{\underline{e}\in J(r)}\textrm{D}(\rho)_{r,d(\underline{e},r),\underline{e}}^{(\underline{\kappa},\delta)\textrm{-(s)s}}\longrightarrow \textrm{SPB}(\rho)_{r,d}^{\delta\textrm{-(s)s}}.
\end{equation*}
We show that he restriction to each component $\Theta_{\underline{e}}: \textrm{D}(\rho)_{r,d(\underline{e},r),\underline{e}}^{(\underline{\kappa},\delta)\textrm{-(s)s}}\longrightarrow \bigcup_{\underline{e'}\leq\underline{e}}\textrm{SPB}(\rho)_{r,d,\underline{e'}}^{\delta\textrm{-(s)s}}$
induces an isomorphism between a (functorialy well defined) dense open subscheme of the stable locus $\mathcal{W}_{\underline{e}}\subset \textrm{D}(\rho)_{r,d(\underline{e},r),\underline{e}}^{(\underline{\kappa},\delta)\textrm{-s}}$ and $\textrm{SPB}(\rho)_{r,d,\underline{e}}^{\delta\textrm{-s}}$ (see Theorem \ref{equivmoduli}). Therefore, $\Theta_{\underline{e}}$ induces a birational surjective and proper morphism $\textrm{D}(\rho)_{r,d(\underline{e},r),\underline{e}}^{(\underline{\kappa},\delta)\textrm{-(s)s}}\twoheadrightarrow \overline{\textrm{SPB}(\rho)_{r,d,\underline{e}}^{\delta\textrm{-s}}}$ when the stable locus is dense inside $\textrm{D}(\rho)_{r,d(\underline{e},r),\underline{e}}^{(\underline{\kappa},\delta)\textrm{-(s)s}}$.

\subsection*{Outline of the paper}
In Section \ref{sectionprelim} we introduce the basic definitions of generalized parabolic swamps and generalized parabolic singular principal $G$-bundles of given type, as well as the semistability conditions. In Section \ref{sectionswamps} we prove the existence of a coarse projective moduli space for generalized parabolic $(\underline{\kappa},\delta)$-(semi)stable swamps of given type. 
The main difficulty here is to find the linearized projective embedding that makes the semistability condition to coincide with the Hilbert-Mumford semistability.
In Section \ref{sectionspb} we prove the existence of a coarse projective moduli space for $(\underline{\kappa},\delta)$-(semi)stable singular principal $G$-bundles. By \cite[Theorem 5.5]{AMC}, this is a direct consequence of the results proved in Section \ref{sectionswamps}. In Section \ref{applicationsection}, we construct the coarse moduli space for descending singular principal bundles over the normalization, as well as the morphism $\Theta$ that relates it with the closure of the stable locus of the moduli space of singular principal bundles over the nodal curve.

\section{Preliminaires}\label{sectionprelim}
%\fancyhead[RO,LE]{}

Let $Y=\coprod_{i=1}^{l}Y_{i}$ be a disjoint union of smooth projective and irreducible curves, $j_{i}:Y_{i}\hookrightarrow Y$ the natural embedding of the $i$th component, $\mathscr{O}_{Y}(1)$ an ample invertible sheaf and $\mathscr{O}_{Y_{1}}(1)=j_{i}^{*}\mathscr{O}_{Y}(1)$ the restriction of $\mathscr{O}_{Y}(1)$ to the component $Y_{i}$. Set $h:=\textrm{deg}(\mathscr{O}_{Y})$ and $h_{i}:=\textrm{deg}(\mathscr{O}_{Y_{i}})$.  Given a coherent sheaf on $Y$, we know that $\mathscr{E}=\bigoplus_{i=1}^{l}j_{i*}(\mathscr{E}|_{i})$, where $\mathscr{E}_{i}:=\mathscr{E}|_{Y_{i}}$. The multirank of $\mathscr{E}$ is defined as the tuple $(r_1,\hdots,r_l)$ (where $r_i=\textrm{rk}(\mathscr{E}_{i})$) while the multidegree is defined as $(d_1,\hdots,d_l)$ (where $d_i=\textrm{deg}(\mathscr{E}_{i})$). If $r\in\mathbb{N}$ and $\textrm{rk}(\mathscr{E}_{i})=r$ for all $i$ (we will say the rank is equal to $r$), then $P_{\mathscr{E}}(n)=\alpha n+r\chi(Y)+d$, where $\alpha=hr$ and $d=\sum_{i=1}^{l}d_{i}$.

\subsection{Generalized parabolic structures}

\begin{definition}\label{genpardef}
Let $r\in\mathbb{N}$, $d\in\mathbb{Z}$ and $\underline{e}:=(e_1,\hdots,e_{\nu})\in\mathbb{N}^{\nu}$ with $e_i\leq r$.
A generalized parabolic locally free sheaf of rank $r$, degree $d$ and type $\underline{e}$ over  $Y$ is a tuple $(\mathscr{E},q_{1},\hdots,q_{\nu})$ where $\mathscr{E}$ is a locally free sheaf of rank $r$  and degree $d$, and $q_{i}$ is a quotient of dimension $e_i$,
$\mathscr{E}(y_{1}^{i})\oplus\mathscr{E}(y_{2}^{i})\twoheadrightarrow R_{i}$,
$\mathscr{E}(y_{j}^{i})$ being the fibre of $\mathscr{E}$ over $y_{j}^{i}$. 
\end{definition}
In order to abreviate the notation we will use the symbol $\underline{q}$ to refer to the tuple $(q_{1},\hdots,q_{\nu})$.
Denote by $R:=\oplus R_{i}$ the total vector space. Since the supports of the divisors $D_{i}$ are disjoint we have $\Gamma(D,\mathscr{E}|_{D})=\bigoplus \Gamma(D_{i},\mathscr{E}|_{D_{i}})= \bigoplus ( \mathscr{E}(y_{1}^{i})\oplus\mathscr{E}(y_{2}^{i}))$. From this, we can form the quotient
$q: =\oplus q_{i}\colon\Gamma(D,\mathscr{E}|_{D})\rightarrow R\rightarrow 0$.

\begin{definition}
Let $(\mathscr{E},\underline{q})$ and $(\mathscr{E}',\underline{q}')$ be generalized parabolic locally free sheaves on $Y$. A homomorphism between them is a tuple $(f,u_{1},\hdots,u_{\nu})$ where $f\colon \mathscr{E}\rightarrow\mathscr{E}'$ is a homomorphism of $\mathscr{O}_{Y}$-modules and $u_{i}\colon R_{i}\rightarrow R_{i}'$ is a homomorphism of vector spaces such that $q'_{i}\circ(f(y_{1}^{i})\oplus f(y_{2}^{i}))=u_{i}\circ  q_{i}$, where $f(y)$ denotes de induced linear map between the fibers at $y\in Y$.
\end{definition}
\begin{notation}
Given a tuple of natural numbers $(e_1,\hdots,e_{\nu})\in\mathbb{N}^{\nu}$, we will denote by 
$I(\underline{e})$ the set $\{i\in\{1,\hdots,\nu\} \textrm{ such that }e_{i}\neq 0\}$
of multitindices of non zero components.
\end{notation}
\begin{definition}\label{kappadeltasemi}
Let $r\in\mathbb{N}$, $d\in\mathbb{Z}$ and $\underline{e}:=(e_1,\hdots,e_{\nu})\in\mathbb{N}^{\nu}$ with $e_i\leq r$. For each $i\in I(\underline{e})$, fix  $\kappa_{i}\in(0,\dfrac{e_{i}}{r})\cap\mathbb{Q}$. Let $(\mathscr{E},\underline{q})$ be a generalized parabolic locally free sheaf of rank $r$, degree $d$ and type $\underline{e}$. We define the $\underline{\kappa}$-parabolic degree for any subsheaf $\mathscr{F}\subseteq\mathscr{E}$ as
\begin{equation*}
\underline{\kappa}\textrm{-}\textrm{pardeg}(\mathscr{F}):= \textrm{deg}(\mathscr{F})-\sum_{i\in I(\underline{e})} \kappa_{i}\dfrac{r}{e_i} \ \textrm{dim} \ q_{i}(\mathscr{F}(y_{1}^{i})\oplus\mathscr{F}(y_{2}^{i}))
\end{equation*}
\end{definition}
\begin{remark}
Formally, we can take as $\kappa_{i}$ any rational number. Taking $\kappa_{i}=\dfrac{e_i}{r}$ we recover the definition given in \cite{Alexander4}. On the other hand, tanking $e_{i}=r$ we recover the definition given in \cite{Alexander1}. Thus, both are particular cases of the one considered in this work. 
\end{remark}

\subsection{Swamps with generalized parabolic structures}

Let $r\in\mathbb{N}$, $d\in\mathbb{Z}$ and $\underline{e}:=(e_1,\hdots,e_{\nu})\in\mathbb{N}^{\nu}$ with $e_i\leq r$.
Fix non negative integers $a,b,c$ and an invertible sheaf $\Lcc$ on $Y$. 
\begin{definition}
A  swamp with generalized parabolic structure of type $(a,b,c,\Lcc,\underline{e})$ rank $r$ and degree $d$ is a triple $(\mathscr{E},\underline{q},\phi)$ where $(\mathscr{E},\underline{q})$ is a generalized parabolic locally free sheaf of rank $r$, degree $d$ and type $\underline{e}$, and $\phi\colon (\mathscr{E}^{\otimes a})^{\oplus b}\rightarrow \textrm{det}(\mathscr{E})^{\otimes c}\otimes \Lcc$ is a non-zero morphism.  
\end{definition}
\begin{notation}
In order to be shorter, we will denote the tuple $(a,b,c,\Lcc,\underline{e})$ that defines the type of a generalized parabolic swamp by the symbol $\mathfrak{tp}$.
\end{notation}

Let $\phi\colon (\mathscr{E}^{\otimes a})^{\oplus b}\rightarrow \textrm{det}(\mathscr{E})^{\otimes c} \otimes\mathscr{L}$ be a swamp on $Y$ and let $(\mathscr{E}_{\bullet},\underline{m})$ be a weighted filtration. For each $\mathscr{E}_{i}$ denote by $\alpha_{i}$ its multiplicity and by $\alpha$ the multiplicity of $\mathscr{E}$. Define the vector 
$\Gamma:=\sum_{1}^{t}m_{i}\Gamma^{(\alpha_{i})}$,
where $\Gamma^{(l)}=(l-\alpha,\overset{\times l}{\hdots},l-\alpha,l,\overset{\times \alpha-l}{\hdots},l)$. Let us denote by $J$ the set
$\{\textrm{multi-indices }I=(i_{1},\hdots,i_{a})|I_{j}\in\{1,\hdots,t+1\}\}$.
Define
\begin{equation*}
\begin{split}
\mu(\mathscr{E}_{\bullet},\underline{m},\phi)&:= - \textrm{min}_{I\in J}\{\Gamma_{\alpha_{i_{1}}}+\hdots +\Gamma_{\alpha_{i_{a}}}|\phi|_{(\mathscr{E}_{i_{1}}\otimes \hdots\otimes \mathscr{E}_{i_{a}})^{\oplus b}}\neq 0\},\\
P_{\underline{\kappa}}(\mathscr{E}_{\bullet},\underline{m})&:=\sum_{i=1}^{s}m_{i}(\underline{\kappa}\textrm{-}\textrm{pardeg}(\mathscr{E})\alpha_{i}-\underline{\kappa}\textrm{-}\textrm{pardeg}(\mathscr{E}_{i})\alpha).
\end{split}
\end{equation*}

\begin{definition}
Let  $\delta\in\mathbb{Q}_{>0}$. For each $i\in I(\underline{e})$, fix  $\kappa_{i}\in(0,\dfrac{e_{i}}{r})\cap\mathbb{Q}$.
A generalized parabolic swamp  $(\mathscr{E},\underline{q},\phi)$ of rank $r$ degree $d$ and type $\mathfrak{tp}=(a,b,c,\Lcc,\underline{e})$ is 
$(\underline{\kappa},\delta)$-(semi)stable if for every weighted filtration $(\mathscr{E}_{\bullet},\underline{m})$ of $\mathscr{E}$, the inequality
$P_{\underline{\kappa}}(\mathscr{E}_{\bullet},\underline{m})+\delta\mu(\mathscr{E}_{\bullet},\underline{m},\phi)(\geq)0$
holds true.
\end{definition}
\begin{remark} \label{finiteset}
Observe that there is a positive integer $A$, depending only on the numerical input data, $r,a,b,c$ and $\mathscr{L}$, such that it is enough to check the $\delta$-semistability condition for weighted filtrations with $m_{i}<A$.
This follows from \cite[Lemma 1.4]{sols} changing ranks by multiplicities.
\end{remark}

Let $S$ be a scheme. Set $S_{D_{i}}:=S\times D_{i}\subset S\times Y$ and let $\pi_{S_i}:S\times D_{i}\rightarrow S$ be the projection onto the first factor. A family of generalized parabolic locally free sheaves parametrized by $S$ is a tuple $(\mathscr{E}_{S},\underline{q}_{S})$ where $\mathscr{E}_{S}$ is a family of locally free sheaves on $Y$ parametrized by $S$ of rank $r$ and degree $d$, and $\underline{q}_{S}=(q_{S1},\hdots,q_{S\nu})$,  $q_{Si}\colon \pi_{Si*}(\mathscr{E}_{S}|_{S_{D_{i}}})\rightarrow R_{i}\rightarrow 0$ being a quotient locally free sheaf of rank $e_{i}$ on $S$.
A family of generalized parabolic swamps is a quadruple $(\mathscr{E}_{S},\underline{q}_{S},\N_{S},\phi_{S})$ where $(\mathscr{E}_{S},\underline{q}_{S})$ is a family of generalized parabolic locally free sheaves of rank $r$ and degree $d$, $\N_{S}$ is an invertible sheaf on $S$, and  $\phi_{S}\colon (\mathscr{E}_{S}^{\otimes a})^{\oplus b}\rightarrow \textrm{det}(\mathscr{E}_{S})^{\otimes c}\otimes \pi_{Y}^{*}\Lcc\otimes\pi_{S}^{*}\N_{S}$ is a morphism of locally free sheaves on $S\times Y$ such that $\phi_{S}|_{\{s\}\times Y}$ is non-zero for all $s\in S$.
Finally, $(\underline{\kappa},\delta)$-(semi)stable families are families which are $(\underline{\kappa},\delta)$-(semi)stable fiberwise.
Then, one can introduce the moduli problem defined by the functor
\begin{equation*}\label{ratuno}
\textbf{SGPS}^{(\underline{\kappa},\delta)-(s)s}_{r,d,\mathfrak{tp}}(S)=\left\{  \begin{array}{l}  \textrm{isomorphism classes of families of} \\ (\underline{\kappa},\delta)\textrm{-(semi)stable generalized parabolic} \\ \textrm{swamps } (\mathscr{E}_{S},\underline{q}_{S},\N_{S},\phi_{S}) \textrm{ parametrized}\\ \textrm{by }S\textrm{ with rank }r, \textrm{ degree }d \textrm{ and type }\mathfrak{tp}
 \end{array}  \right\}.
\end{equation*}

\subsection{Singular principal $G$-bundles with generalized parabolic structures}

Let $G$ be a semisimple linear algebraic group and let $\rho:G\hookrightarrow \textrm{SL}(V)$ be a faithful representation.

\begin{definition}
A singular principal $G$-bundle over $Y$ is a pair $(\mathscr{E},\tau)$ where $\mathscr{E}$ is a locally free sheaf and $\tau:S^{\bullet}(V\otimes  \mathscr{E})^{G}\rightarrow \mathscr{O}_{Y}$ is a non-trivial morphism of $\mathscr{O}_{Y}$-algebras.
\end{definition}

\begin{definition}\label{defspbgps}
Let $r\in\mathbb{N}$, $d\in\mathbb{Z}$ and $\underline{e}:=(e_1,\hdots,e_{\nu})\in\mathbb{N}^{\nu}$ with $e_i\leq r$.
A singular principal $G$-bundle with a generalized parabolic structure over $Y$ of rank $r$, degree $d$ and type $\underline{e}$ is a triple $(\mathscr{E},\tau,\underline{q})$ where $(\mathscr{E},\underline{q})$ is a generalized parabolic locally free sheaf of rank $r$, degree $d$ and type $\underline{e}$, and $(\mathscr{E},\tau)$  is a singular principal $G$-bundle.
\end{definition}

\begin{definition}
Let $(\mathscr{E},\tau,\underline{q})$ and $(\G,\lambda,\underline{p})$ be singular principal $G$-bundles with generalized parabolic structure on $Y$. A morphism between them is a morphism of $\mathscr{O}_{Y}$-modules $f:\mathscr{F}\rightarrow\G$ compatible with both structures. The isomorphisms are the obvious ones.
\end{definition}

Following \cite[Theorem 5.5]{AMC}, we can assign to any singular principal $G$-bundle a swamp of type $(a,b,0,\mathscr{O}_{Y})$ for certain naural numbers $a,b$ that depends only on the numerical input data, 
\begin{equation}\label{swamp-bundle}
\begin{split}
\left\{   \begin{array}{l}  \textrm{isomorphism classes} \\ \textrm{of singular principal} \\ G\textrm{-bundles}  \end{array} \right\} &\rightarrow \left\{   \begin{array}{l}  \textrm{isomorphism classes} \\ \textrm{of swamps} \\ \textrm{of type }(a, b,0, \mathscr{O}_{Y})   \end{array} \right\}, \\
(\mathscr{E},\tau,\underline{q})&\mapsto (V\otimes \mathscr{E},\varphi_{\tau},\underline{q})
\end{split}
\end{equation}
this map being injective. Thus, we can define, for any weighted filtration $(\mathscr{E}_{\bullet},\underline{m})$, the semistability function $\mu(\mathscr{E}_{\bullet},\underline{m},\tau)$ as $\mu(\mathscr{E}_{\bullet},\underline{m},\varphi_{\tau})$ (see \cite[Definition 6.1]{AMC}).
\begin{definition}\label{gpsstable}

Let $r\in\mathbb{N}$, $d\in\mathbb{Z}$, $\underline{e}:=(e_1,\hdots,e_{\nu})\in\mathbb{N}^{\nu}$ with $e_i\leq r$, and $\delta\in\mathbb{Q}_{>0}$. For each $i\in I(\underline{e})$, fix  $\kappa_{i}\in(0,\dfrac{e_{i}}{r})\cap\mathbb{Q}$.
A generalized parabolic singular principal $G$-bundle of rank $r$ degree $d$ and type $\underline{e}$,  $(\mathscr{E},\underline{q},\tau)$, is $(\underline{\kappa},\delta)$-(semi)stable if for every weigted filtration $(\mathscr{E}_{\bullet},\underline{m})$ of $\mathscr{E}$, the inequality
$P_{\underline{\kappa}}(\mathscr{E}_{\bullet},\underline{m})+\delta\mu(\mathscr{E}_{\bullet},\underline{m},\tau)(\geq)0$
holds true.
\end{definition}
Then, one can define a family as in the case of swamps and introduce the moduli problem defined by the functor
\begin{equation}\label{modulifunctornormalization}
\textbf{SPBGPS}(\rho)_{r,d,\underline{e}}^{(\underline{\kappa},\delta)\textrm{-(s)s}}(S)=\left\{   \begin{array}{l}\textrm{isomorphism classes of families}\\ \textrm{of } (\underline{\kappa},\delta)\textrm{-(semi)stable} \textrm{ singular} \\ 
\textrm{principal }G\textrm{-bundles with} \\ 
\textrm{generalized parabolic structure} \\ 
\textrm{on $Y$ parametrized } \textrm{by $S$ with}\\
\textrm{rank $r$ degree }d \textrm{ and type }\underline{e}
\end{array}  \right\}.
\end{equation}

\subsection{Some calculations in geometric invariant theory}\label{gitcalculations}

Recall that a basis $\underline{u}:=(u_1,\hdots,u_p)$ of the vector space $U$, together with a vector $\underline{\gamma}=(\gamma_{1},\hdots, \gamma_{p})\in\mathbb{N}^{p}$ such that $\gamma_{1}\leq\hdots\leq \gamma_{p}$ and $\sum_{i=1}^{p}\gamma_{i}=0$, defines a one parameter subgroup $\lambda(\underline{u},\underline{\gamma}):\mathbb{C}\rightarrow \textrm{SL}(U)$. Conversely, every one parameter subgroup of $\textrm{SL}(U)$ arises in this way (see \cite[Example 1.5.1.12]{Asch}).
Furthermore, every one parameter subgroup of  $\textrm{SL}(U)$ determines a weighted flag $(U_{\bullet},\underline{m})$ of $U$ and every weighted flag arises in this way as well. It turns out that the Hilbert-Mumford function, $\mu(-,\lambda)$ depends only on the associated weighted flag of $\lambda$ and not on $\lambda$ itself (see \cite[Proposition 1.5.1.35, Example 1.5.1.36]{Asch}). 

We derive the explicit expression of the Hilbert-Mumford criterion (see \cite[Theorem 2.1, Proposition 2.3]{mum-git}) in some situations that will be important for our proposes. Similar calculations can be found along \cite{Asch}, so we will skip some details.

\subsubsection{Example 1}

Let $p, r$ be integers such that $1\leq e\leq p-1$. Let $\G r:=\textrm{Grass}_{e}(U^{\oplus 2})$ be the Grassmannian of $e$-dimensional quotients of $U^{\oplus 2}$, $U$ being a $p$-dimensional vector space, and let $N$ be positive integer. The Grassmannian can be embedded into the projective space through the Pl\"ucker embedding $\iota\colon  \G r \hookrightarrow \textbf{P}(\wedge^{e}U^{\oplus 2})$.
The group $\textrm{SL}(U)$ acts on both spaces through the diagonal $\delta:\textrm{SL}(U)\hookrightarrow \textrm{SL}(U^{\oplus 2})$ in the obvious way, and $\iota$ is $\textrm{SL}(U)$-equivariant. If $\mathscr{O}(1)$ is the tautological invertible sheaf on $\textbf{P}(\wedge^{e}U^{\oplus 2})$, then $\Lcc:=\iota^{*}\mathscr{O}(1)$ is a $\textrm{SL}(U)$-linearized very ample invertible sheaf. Let us compute the semistability function of points in $\G r$ with respect to $\Lcc$. 

Let $\{u_{1},\hdots,u_{p}\}$ be a basis of the vector space $U$. Then, a basis of $\wedge^{e}U^{\oplus 2}$ is given by the vectors 
$
u_{I,J}:=(u_{i_{1}},0)\wedge\hdots\wedge(u_{i_{l}},0)\wedge(0,u_{j_{1}})\wedge\hdots\wedge(0,u_{j_{e-l}}).
$
Let $\lambda\colon \GG\rightarrow \textrm{SL}(U)$ be a one parameter subgroup. Fix a basis $\underline{u}=\{u_{1},\hdots,u_{p}\}$ and integers $\gamma_{1}\leq\hdots\leq\gamma_{p}$ such that $\lambda=\lambda(\underline{u},\underline{\gamma})$. Then, 
we have
\begin{equation*}
%\begin{split}
\mu^{\Lcc}(\tau,\lambda(\underline{u},\underline{\gamma}))=\sum_{i=1}^{s}i \ \textrm{dim} ( \textrm{Ker}(\tau))-p \ \textrm{dim}(\textrm{Ker}(\tau)\cap(U_{i}\oplus U_{i}))%\\
=\sum_{i=1}^{s} p \ \textrm{dim}\tau(U_{i}\oplus U_{i})-i e,
%\end{split}
\end{equation*}
where $(U_{\bullet},\underline{m})$ is the weighted filtration associated to $\lambda$.

\subsubsection{Example 2}\label{ex2}

Let $Y_{1},\hdots, Y_{l}$ be smooth projective connected curves, and consider their disjoint union, $Y:=\bigsqcup Y_i$. Let $\N_{1},\hdots, \N_{l}$ be invertible sheaves on $Y_{1},\hdots, Y_{l}$ respectively and denote  by $\N:=\bigoplus \N_{i}$ the corresponding invertible sheaf on $Y$. Let $r, \ n\in\mathbb{N}$ and let $U$ be a vector space of dimension $p>r$. Consider now, for each $i$, the projective space given by
%\begin{equation*}
$
\mathbb{G}_{1,\N}^{i}:=\mathbf{P}(\textrm{Hom}(\bigwedge^{r}U,H^{0}(Y_{i},\N_{i}(rn)))^{\vee}),
$
%\end{equation*}
and define $\mathbb{G}_{1,\N}=\mathbb{G}_{1,\N}^{1}\times\hdots\times\mathbb{G}_{1,\N}^{l}$. Let $b_{1},\hdots, b_{l}\in\mathbb{N}$ and consider the very ample invertible sheaf on $\mathbb{G}_{1}$ given by
$\Lcc:=\pi_{1}^{*}\mathscr{O}_{\mathbb{G}_{1,\N}^{1}}(b_{1})\otimes\hdots \otimes \pi_{l}^{*}\mathscr{O}_{\mathbb{G}_{1,\N}^{l}}(b_{l})$
with the obvious $\textrm{SL}(U)$-linearization. For the sake of clarity, we will use the symbol $\Lcc_{i}$ to denote the invertible sheaf $\mathscr{O}_{\mathbb{G}_{1,\N}^{i}}(1)$. 
Clearly
%\begin{equation*}
$
\mu^{\Lcc}([g],\lambda)=\sum_{i=1}^{l}b_{i}\mu^{\pi_{i}^{*}\Lcc_{i}}([g],\lambda)=\sum_{i=1}^{l}b_{i}\mu^{\Lcc_{i}}([g_{i}],\lambda),
$
%\end{equation*}
$[g_{i}]$ being the $i$-th component of $[g]$. Therefore the calculation of the semistability function of points of $\mathbb{G}_{1,\N}$ with respect to $\Lcc$ is reduced to the calculation of the semistability function of points of $\mathbb{G}_{1,\N}^{i}$ with respect to $\Lcc_{i}$.
Let $\mathscr{E}$ be a  locally free quotient sheaf of rank $r$
$q\colon U\otimes\mathscr{O}_{Y}(-n)\rightarrow\mathscr{E}\rightarrow 0$
whose determinant is isomorphic to $\mathscr{N}$. Restricting to the $i$-th component, twisting by $n$, taking the $r$-th exterior power and taking global sections we find the morphism
%\begin{equation*}
$
H^{0}(\wedge^{r}(q_{i}(n)))\colon \wedge^{r}U\rightarrow H^{0}(Y,\mathscr{N}_{i}(rn)),
$
%\end{equation*}
whose equivalence class defines a point $[H^{0}(\wedge^{r}(q_{i}(n)))]\in\mathbb{G}_{1,\N}^{i}$. 
Now, a short calculation shows that
\begin{equation*}
\mu^{\Lcc_{j}}( [H^{0}(\wedge^{r}(q_{j}(n)))] ,\lambda)=\sum_{i=1}^{s}m_{i}(\textrm{rk}(\mathscr{E}_{i}|_{Y_{j}})p-r \textrm{dim}(U_{i})),
\end{equation*}
$(U_{i},m_{i})$ being the $i$th term of the weighted filtration associate to $\lambda$ and $\mathscr{E}_{i}|_{Y_{j}}$ the restriction to $Y_{j}$ of the saturated subsheaf generated by $U_{i}$.

\subsubsection{Example 3}
Consider the same situation as in Example 2. Let $\Lcc$ be an invertible sheaf on $Y$, $U$ a $p$-dimensional vector space and $a,b,c, n\in\mathbb{N}$. Given an invertible sheaf $\mathscr{N}$ on $Y$ we define the projective space
%\begin{equation*}
$\mathbb{G}_{2,\mathscr{N}}= \mathbf{P}(\textrm{Hom}(U_{a,b},H^{0}(Y,\mathscr{N}^{\otimes c}\otimes\Lcc (na)))^{\vee}),$
%\end{equation*}
where $U_{a,b}:=(U^{\otimes a})^{\oplus b}$.
Consider the pair $(q,\phi)$ given by a  locally free  quotient sheaf of rank $r$,  $q\colon U\otimes\mathscr{O}_{Y}(-n)\rightarrow\mathscr{E}$, whose determinant is isomorphic to $\mathscr{N}$ and a morphism $\phi\colon (\mathscr{E}^{\otimes a})^{\oplus b}\rightarrow\mathscr{N}^{\otimes c}\otimes\Lcc$. Let $\Delta\colon U_{a,b}\hookrightarrow U_{a,b}^{\oplus l}$ be the diagonal linear map, and consider the morphism 
%\begin{equation*}
$H^{0}((q(n)^{\otimes a})^{\oplus b})\circ\Delta\colon U_{a,b}\rightarrow H^{0}(Y,(\mathscr{E}^{\otimes a})^{\oplus b}\otimes\mathscr{O}_{Y}(na)).$
%\end{equation*}
Twisting $\phi$ by $\mathscr{O}_{Y}(na)$, we get
%\begin{equation*}
$H^{0}(\phi(na)):H^{0}(Y,(\mathscr{E}^{\otimes a})^{\oplus b}\otimes\mathscr{O}_{Y}(na))\rightarrow H^{0}(Y,\mathscr{N}^{\otimes c}\otimes\Lcc(na)).$
%\end{equation*}
Composing both morphisms we get a point in $\mathbb{G}_{2,\mathscr{N}}$,
\begin{equation}\label{uf}
[H^{0}(\phi(na))\circ H^{0}((q(n)^{\otimes a})^{\oplus b})\circ\Delta]\colon U_{a,b}\rightarrow H^{0}(Y,\mathscr{N}^{\otimes c}\otimes\Lcc(na))]\in \mathbb{G}_{2,\mathscr{N}}.
\end{equation}
Set $p=\textrm{dim}(U)$ and let $\underline{u}=(u_{1},\hdots,u_{p})$ be a basis of $U$. For any multiindex $I=(i_{1},\hdots,i_{a})$ with $i_{j}\in\{1,\hdots,p\}$ define $u_{I}:=u_{i_{1}}\otimes\hdots\otimes u_{i_{a}}$ and $u_{I}^{k}:=(0,\hdots,0,\overset{k)}{u_{I}},0,\hdots,0).$
Then the elements $u_{I}^{k}$ form a basis of $U_{a,b}$ and the group $\textrm{SL}(U)$ acts on $\mathbb{G}_{2,\mathscr{N}}$ in the obvious way.
We want to compute the semistability function for points $T\in\mathbb{G}_{2,\N}$ of the form (\ref{uf}) with respect to the natural $\textrm{SL}(U)$-linearization of $\mathscr{O}_{\mathbb{G}_{2,\mathscr{N}}}(1)$.
Let $\lambda\colon \GG \rightarrow \textrm{SL}(U)$ be a one parameter subgroup. Then there exists a basis $u_{1},\hdots,u_{p}$ of $U$ and integers $\gamma_{1}\leq\hdots\leq\gamma_{p}$ with $\sum\gamma_{i}=0$ such that
%\begin{equation*}
$\lambda(z)u_{i}=z^{\gamma_{i}}u_{i},\forall z\in\GG$.
%\end{equation*}
For any multiindex $I=(i_{1},\hdots,i_{a})$ consider $u_{I}$ and define $\gamma_{I}=\gamma_{i_{1}}+\cdots +\gamma_{i_{a}}$. Then $\lambda\colon \GG \rightarrow \textrm{SL}(U)$ acts by
$\lambda(z)\bullet u_{I}^{k}=z^{\gamma_{I}}\bullet u_{I}^{k}, \forall z\in \GG$,
and we have $\mu([T],\lambda)=-\textrm{min}\{\gamma_{I}|T(u_{I}^{k})\neq 0\}$.
Given a multiindex $I=(i_{1},\hdots,i_{a})$ we want to compute $\gamma_{I}=\gamma_{i_{1}}+\cdots+\gamma_{i_{a}}$ 
for $\gamma=(i-p,\hdots,i-p,i,\hdots,i)$. Set $\nu(I,i):=\#\{j|i_{j}\leq i\}$. Then $i_{1},\hdots,i_{\nu(I,i)}\leq i$ and $i_{\nu(I,i)+1},\hdots,i_{a}>i$, so
$\gamma_{I}=(i-p)\nu(I,i)+i(a-\nu(I,i))=ia-\nu(I,i)p$.
A short calculation shows
\begin{equation*}
\mu([T],\lambda)=\sum_{i=1}^{s}m_{i}(\nu(I,\textrm{dim} U_{i})p-\textrm{dim} U_{i}a),
\end{equation*}
$(U_{\bullet},\underline{m})$ being the weighted flag associated to $\lambda$  and $I=(i_{1},\hdots,i_{a})$ is the multiindex giving the minimum of the semistability function.

\section{Moduli space for generalized parabolic swamps}\label{sectionswamps}
%\fancyhead[RO,LE]{}

Let $r\in\mathbb{N}$, $d\in\mathbb{Z}$, $\underline{e}:=(e_1,\hdots,e_{\nu})\in\mathbb{N}^{\nu}$ with $e_i\leq r$ and $\delta\in\mathbb{Q}_{>0}$. For each $i\in I(\underline{e})$ fix $\kappa_{i}\in(0,\dfrac{e_{i}}{r})\cap\mathbb{Q}$. Fix non negative integers $a,b,c$ and an invertible sheaf $\Lcc$ on $Y$. Recall that $h:=\textrm{deg}(\mathscr{O}_{Y}(1))$ and $h_{i}:=\textrm{deg}(\mathscr{O}_{Y}(1)_{|_{Y_{i}}})$.
It will be assumed that these data are fixed once and for all along this section.

The main result of this section is Theorem \ref{theorem E} which shows the existence of a coarse projective moduli space for $(\underline{\kappa},\delta)$-(semi)stable swamps with generalized parabolic structure of given type $\mathfrak{tp}=(a,b,c,\Lcc,\underline{e})$ and with rank and degree equal to $r$ and $d$ respectively. 
In order to do so, we have to consider the rigidified functor
\begin{equation}\label{rigfunctorswamps}
^{rig}\textbf{SGPS}^{n}_{r,d,\mathfrak{tp}}(S)=\left\{   \begin{array}{l} \textrm{isomorphism classes of tuples }(\mathscr{E}_{S},\underline{q}_{S},\phi_{S},g_{S}) \\
\textrm{where } (\mathscr{E}_{S},\phi_{S}) \textrm{ is a family of swamps} \\ 
 \textrm{parametrized by }S\textrm{ with rank $r$ and degree $d$}  \\
(\mathscr{E}_{S},\underline{q}_{S}) \textrm{ is a family of generalized parabolic}  \\
\textrm{locally free sheaves }\textrm{ and }g_{S}:U\otimes\mathscr{O}_{S}\rightarrow\pi_{S*}\mathscr{E}_{S}(n)\\
\textrm{is a morphism such that the induced morphism } \\
U\otimes\mathscr{O}_{Y\times S}(-n)\rightarrow\mathscr{E}_{S}\textrm{ is surjective }
 \end{array}\right\},
\end{equation}
where $n\in\mathbb{N}$, $U:=\mathbb{C}^{P(n)}$, $P(n)=\alpha n+r\chi(\mathscr{O}_{Y})+d$ and $\alpha=hr$.

\subsection{Boundedness for generalized parabolic swamps}

Let us denote by $E_{d,r}$ the family of locally free sheaves on $Y$ of rank $r$ and degree $d$. %uniform multirank $r$ and degree $d$. 
Recall that a family of sheaves $E\subset E_{d,r}$ on $Y$ is bounded if and only if there is a natural number $n_{0}\in\mathbb{N}$ such that for all $n\geq n_{0}$ and all locally free sheaves $\mathscr{E}\in E$, $h^{1}(Y,\mathscr{E}(n))=0$ and $\mathscr{E}(n)$ is globally generated.
Boundedness for locally free sheaves appearing in $(\underline{\kappa},\delta)$-(semi)stable swamps with generalized parabolic structures (Proposition \ref{boundparabolicSwamp2}) 
will follow from the next observation.

Let $\mathscr{E}$ be a locally free sheaf over $Y$ and let $(\mathscr{E}_{\bullet},\underline{m})$ be a weighted filtration, with
$\mathscr{E}_{\bullet}\equiv(0)\subset\mathscr{E}_{1}\subset\hdots\subset\mathscr{E}_{s} \subset \mathscr{E}$.
Consider a partition of the multitindex $I:=(1,2,\hdots,s)$,
$I=I_{1}\sqcup I_{2}$,
let us say $I_{1}=(i_{1},\hdots,i_{t})$ and $I_{2}=(k_{1},\hdots,k_{s-t})$. Then, a simple calculation (see \cite[Lemma 1.6]{sols} for the connected case) shows that
\begin{equation}\label{lema bound}
\begin{split}
(\sum_{i=1}^{s}m_{i})a(\alpha-1) &\geq \mu(\mathscr{E}_{\bullet},\underline{m},\phi)\geq-(\sum_{i=1}^{s}m_{i})a (\alpha-1),\\
\mu(\mathscr{E}_{\bullet},\underline{m},\phi) &\geq\mu(\mathscr{E}^{1}_{\bullet},\underline{m_{1}},\phi)-(\sum_{i=1}^{s-t}m_{2,i})a(\alpha-1),
\end{split}
\end{equation}
where $\mathscr{E}^{1}_{j}=\mathscr{E}_{i_{j}}$.
The following results are important direct consequences of Equation (\ref{lema bound}).

\begin{proposition}\label{lemacarra4}
A generalized parabolic swamp $(\mathscr{E},\underline{q},\phi)$ is $(\underline{\kappa},\delta)$-(semi)stable if and only if for any weighted filtration $(\mathscr{E}_{\bullet},\underline{m})$, such that $\emph{par}\mu(\mathscr{E}_{i})\geq \emph{par}\mu(\mathscr{E})-C_{1}$, where $C_{1}=a\delta +r\nu$, the inequality
$P_{\underline{\kappa}}(\mathscr{E}_{\bullet},\underline{m})+\delta\mu(\mathscr{E}_{\bullet} ,\underline{m},\phi)(\geq)0$
holds true.
\end{proposition}
\begin{proof}
Let $(\mathscr{E}_{\bullet},\underline{m})$ be a weighted filtration such that $\textrm{par}\mu(\mathscr{E}_{i})< \textrm{par}\mu(\mathscr{E})-C_{1}$ for all $i$. Since $\underline{\kappa}\textrm{-pardeg}(\mathscr{E}_{i})\alpha-\underline{\kappa}\textrm{-pardeg}(\mathscr{E})\alpha_{i}< -C_{1}\alpha\alpha_{i}$, Equation (\ref{lema bound}) implies $P_{\underline{\kappa}}(\mathscr{E}_{\bullet},\underline{m})+ \delta\mu(\mathscr{E}_{\bullet}\underline{m},\phi)\geq 0$.
\end{proof}

\begin{proposition}\label{boundparabolicSwamp2}

The family of locally free sheaves  of degree $d$ and rank $r$ appearing in
$(\underline{\kappa},\delta)$-(semi)stable swamps with generalized parabolic structure is bounded.
\end{proposition}
\begin{proof}
By  Equation (\ref{lema bound}) and a simple calculation, it follows that
$\mu(\mathscr{E}')\leq\mu(\mathscr{E})+a\delta+r\nu$.
Then, we conclude by \cite[Lemma 2.5]{FGA}.
\end{proof}

\begin{remark}\label{coletascabron}
Let $C'_{1}=\alpha C_{1}$. Note that if $\textrm{deg}(\mathscr{E})\leq 0$ then $\textrm{deg}(\mathscr{E}')\leq \mu(\mathscr{E})+C'_{1}$, and if $\textrm{deg}(\mathscr{E})>0$ then $\textrm{deg}(\mathscr{E}')\leq \textrm{deg}(\mathscr{E})+C'_{1}$. In both cases the degree of any subsheaf $\mathscr{E}'\subset\mathscr{E}$ is bounded by a constant depending only on $a,\delta,r,h,\nu, d$. This in particular means that for any locally free sheaf $\mathscr{E}$ of rank $r$ and degree $d$ appearing in a $(\underline{\kappa},\delta)$-semistable swamp with generalized parabolic structure of type $(a,-,-,-,-)$ (this means that the first component is fixed and equal to $a$ but the others are left to be free) we have that $\textrm{deg}(\mathscr{E}|_{Y_{i}})$ is bounded from below and above by constants depending only on $a,\delta,\alpha,\nu, d$ which we will denote by $A_{-}(a,\delta,r,h,\nu, d)$ and $A_{+}(a,\delta,r,h,\nu, d)$, or just by $A_{-}$ and $A_{+}$ if there is no confusion.
\end{remark}

\subsection{The Gieseker space and map}\label{giesekerpoint}
Our goal now is to construct the Gieseker space together with the Gieseker map, and to construct a representative for the moduli functor given in Equation (\ref{rigfunctorswamps}). We will assume that $e_{i}\neq 0$ for each $i=1,\hdots \nu$. If $e_{i}=0$ for some index $i$, we will only have to drop the corresponding Grassmannian in Equation (\ref{Z-space}) and Equation (\ref{giesmap}) below.

\subsubsection{The parameter space}\label{parameterparten}

Let $H$ be an effective divisor of degree $h$ in $Y$ such that $\mathscr{O}_{Y}(H)\simeq \mathscr{O}_{Y}(1)$.
By Proposition \ref{boundparabolicSwamp2} we know that there exists a natural number $n_{0}\in\mathbb{N}$ such that for every $n\geq n_{0}$ and every $(\underline{\kappa},\delta)$-(semi)stable generalized parabolic swamp of type $\mathfrak{tp}=(a,b,c,\Lcc,\underline{e})$ of rank $r$ and degree $d$ we have
$
H^{1}(Y,\mathscr{E}(n))=H^{1}(Y,\textrm{det}(\mathscr{E}(rn)))=H^{1}(Y,\textrm{det}(\mathscr{E})^{\otimes c}\otimes\Lcc\otimes\mathscr{O}_{Y}(an))=0
$
and the locally free sheaves  $\mathscr{E}(n),\textrm{det}(\mathscr{E}(rn)),\textrm{det}(\mathscr{E})^{\otimes c}\otimes\Lcc\otimes\mathscr{O}_{Y}(an)$ are globally generated.
Fix $n\geq n_{0}$ as above,
%Let $n_{1}\in\mathbb{N}$ be as in Theorem \ref{sectionalsemi}, fix $n>\textrm{max}\{n_{0},n_{1}\}$ 
and $\underline{d}=(d_{1},\hdots,d_{l})\in\mathbb{N}^{l}$ with $d=\sum_{i=1}^{l}d_{i}$, and set $p=r\chi(\mathscr{O}_{Y})+ d+\alpha n$ (recall $\alpha=hr$). Let $U$ be the vector space $\mathbb{C}^{\oplus p}$. We will use the notation $U_{a,b}$ for $(U^{\otimes a})^{\oplus b}$. Denote by $\mathcal{Q}^{0}$ the quasi-projective scheme parametrizing equivalence classes of quotients $\q\colon U\otimes\pi_{Y}^{*}\mathscr{O}_{Y}(-n)\rightarrow\mathscr{E}$ where $\mathscr{E}$ is a locally free sheaf of uniform multirank $r$ and multidegree $\underline{d}=(d_{1},\hdots,d_{l})$ on $Y$ and such that the induced map $U\rightarrow H^{0}(Y,\mathscr{E}(n))$ is an isomorphism. On $\Q\times Y$, we have the universal quotient
%\begin{equation*}
$
\q_{\Q}\colon U\otimes\pi_{Y}^{*}\mathscr{O}_{Y}(-n)\rightarrow \mathscr{E}_{\Q}.
$
%\end{equation*}
Since $n>n_{0}$, the sheaf
%\begin{equation*}
$
\mathcal{H}:=\mathcal{H}om_{\mathscr{O}_{\Q}} (U_{a,b}\otimes\mathscr{O}_{\Q}, \pi_{\Q*}(\textrm{det}(\mathscr{E}_{\Q})^{\otimes c}\otimes\pi_{Y}^{*}\Lcc \otimes \pi_{Y}^{*}\mathscr{O}_{Y}(na)))
$
%\end{equation*}
is locally free.
Consider the corresponding projective bundle $\pi'\colon \mathfrak{h}=\Pp(\mathcal{H}^{\vee}) \rightarrow\Q$ and let
$\q_{\mathfrak{h}}\colon U\otimes\pi_{Y}^{*}\mathscr{O}_{Y}(-n)\rightarrow \mathscr{E}_{\mathfrak{h}}$
be the pullback of the universal locally free sheaf to $\mathfrak{h}\times Y$. Now, the tautological invertible quotient on $\mathfrak{h}$,
$\pi^{'*}\mathcal{H}^{\vee}\rightarrow \mathscr{O}_{\mathfrak{h}}(1)\rightarrow 0$,
induces a morphism on $\mathfrak{h}\times Y$,
$s_{\mathfrak{h}}:U_{a,b}\otimes \mathscr{O}_{\mathfrak{h}}\rightarrow \textrm{det}(\mathscr{E}_{\mathfrak{h}})^{\otimes c}\otimes \pi_{Y}^{*}\Lcc\otimes\pi_{Y}^{*}\mathscr{O}_{Y}(na) \otimes\pi_{\mathfrak{h}}^{*} \mathscr{O}_{\mathfrak{h}}(1)$.
From the universal quotient we get a surjective morphism
%\begin{equation*}
$(\mathfrak{q}_{\mathfrak{h}}^{\otimes a})^{\oplus b}:U_{a,b}\otimes\pi_{Y}^{*}\mathscr{O}_{Y}(-na)\rightarrow (\mathscr{E}_{\mathfrak{h}}^{\otimes a})^{\oplus b}.$
%\end{equation*}
Denoting by $\K$ its kernel, we get a diagram
$$
\xymatrix{
0\ar[r] & \K\ar[r]\ar@{-->}[rd] & U_{a,b}\otimes\pi_{Y}^{*}\mathscr{O}_{Y}(-na) \ar[r] \ar[d]^{s_{\mathfrak{h}}\otimes \pi_{Y}^{*}id_{\mathscr{O}_{Y}(-na)}} & (\mathscr{E}_{\mathfrak{h}}^{\otimes a})^{\oplus b}\ar[r] & 0 \\
   &  &  \textrm{det}(\mathscr{E}_{\mathfrak{h}})^{\otimes c}\otimes \pi_{Y}^{*}\Lcc\otimes\pi_{\mathfrak{h}}^{*} \mathscr{O}_{\mathfrak{h}}(1) & &
}
$$
From \cite[Lemma 3.1]{sols}, it follows that
there is a closed subscheme $\mathfrak{G}\subset \mathfrak{h}$ over which $s_{\mathfrak{h}}\otimes \pi_{Y}^{*}\textrm{id}_{\mathscr{O}_{Y}(-na)}$ factorizes through a morphism
%\begin{equation*}
$
\phi_{\mathfrak{G}}:(\mathscr{E}_{\mathfrak{G}}^{\otimes a})^{\oplus b}\rightarrow \textrm{det}(\mathscr{E}_{\mathfrak{G}})^{\otimes c}\otimes \pi_{Y}^{*}\Lcc\otimes\pi_{\mathfrak{G}}^{*} \mathfrak{N_{\mathfrak{G}}},
$
%\end{equation*}
$\pi_{\mathfrak{G}}^{*} \mathfrak{N_{\mathfrak{G}}}$ being the pullback of the restriction of $\mathscr{O}_{\mathfrak{h}}(1)$ to $\mathfrak{G}$. Then, on the scheme $\mathfrak{G}\times Y$ we have a family of swamps $(\mathscr{E}_{\mathfrak{G}},\mathfrak{N}_{\mathfrak{G}},\phi_{\mathfrak{G}})$ parametrized by $\mathfrak{G}$. 
In order to include the parabolic structure, we need to consider the Grassmannian $\mathscr{G} r_{i}:=\textrm{Grass}_{e_{i}}(U^{\oplus 2})$ of $e_{i}$ dimensional quotients of $U^{\oplus 2}$. Recall that $\nu$ is the number of nodes of the curve, so that we have $\nu$ divisors, $D_{i}=y_{1}^{i}+y_{2}^{i}$, in the normalization $Y$. Define, 
\begin{equation}\label{Z-space}
Z:=\mathfrak{G}\times \mathscr{G} r_{1}\times\cdots\times\mathscr{G} r_{\nu}
\end{equation} 
and denote by $c_{i}:Z\rightarrow\mathscr{G} r_{i}$ the $i$th projection. Consider the pullback of the universal quotient of the Grassmannian $\mathscr{G} r_{i}$ by the projection $c_{i}$,
$q_{Z}^{i}: U^{\oplus 2}\otimes\mathscr{O}_{Z}\rightarrow R_{Z}$,
and take the direct sum
$q_{Z}: U^{\oplus 2\nu}\otimes\mathscr{O}_{Z}\rightarrow \bigoplus_{1}^{\nu} R_{Z}$.
Consider now the two natural projections
$
\mathfrak{G}\times Y\rightarrow  \mathfrak{G}, \  Z\times Y\rightarrow Z.
$
Denote by $\mathfrak{N}_{Z}$ the pullback of $\mathfrak{N}_{\mathfrak{G}}$ to $Z$, and by $\mathfrak{q}_{Z}$, $\mathscr{E}_{Z}$ and $\phi_{Z}$ the pullbacks of the corresponding objects over $\mathfrak{G}\times Y$ to $Z\times Y$. Consider the morphisms $\pi^{i}: Z\times\{y^{i}_{1}, y^{i}_{2}\}\rightarrow Z\times\{x_{i}\}\simeq Z$.
For each $i$, there are quotients $f_{i}:U^{\oplus 2}\times\mathscr{O}_{Z}\rightarrow \pi^{i}_{*}(\mathscr{E}_{Z}|_{y^{i}_{1},y^{i}_{2}})$ and we can form $f\colon=\bigoplus f_{i}:U^{\oplus 2\nu}\times\mathscr{O}_{Z}\rightarrow \bigoplus\pi^{i}_{*}(\mathscr{E}_{Z}|_{y^{i}_{1},y^{i}_{2}})$.
Consider the following diagram,
$$
\xymatrix{
0\ar[r]& \textrm{Ker}(f)\ar[r]\ar@{.>}[rd]^{q'} & 
U^{\oplus 2\nu}\times\mathscr{O}_{Z} \ar[r]^{f}\ar[d]^{q_{Z}}& \bigoplus\pi^{i}_{*}(\mathscr{E}_{Z}|_{y^{i}_{1},y^{i}_{2}}) \ar[r]& 0 \\
 & &  \bigoplus_{1}^{\nu} R_{Z}  & &
}
$$
Denote by $\mathfrak{I}_{\underline{d}}\subset Z$ the closed subscheme given by the zero locus of the morphism $q'$ (see \cite[lemma 3.1]{sols} again). Then the restriction of $q_{Z}$ to $\mathfrak{I}_{\underline{d}}$ factorizes through 
\begin{equation*}
q_{\mathfrak{I}_{\underline{d}}}\colon\bigoplus_{1}^{\nu}\pi_{*}^{i}(\mathscr{E}_{Z}|_{y^{i}_{1},y^{i}_{2}})|_{\mathfrak{I}_{\underline{d}}}= \bigoplus_{1}^{\nu}\pi^{i}_{\mathfrak{I}_{\underline{d}}*}(\mathscr{E}_{\mathfrak{I}_{\underline{d}}}|_{y^{i}_{1},y^{i}_{2}}) \rightarrow \bigoplus_{1}^{\nu} R_{Z}|_{\mathfrak{I}_{\underline{d}}}=\bigoplus_{1}^{\nu} R_{\mathfrak{I}_{\underline{d}}}.
\end{equation*}
Since $f$ and $q_{Z}$ are diagonal morphisms we deduce that $q_{\mathfrak{I}_{\underline{d}}}$ is also diagonal. Therefore $q_{\mathfrak{I}_{\underline{d}}}$ is determined by $\nu$ morphisms
$q_{\mathfrak{I}_{\underline{d}}}^{i}:\pi^{i}_{\mathfrak{I}_{\underline{d}}*}(\mathscr{E}_{\mathfrak{I}_{\underline{d}}}|_{y^{i}_{1},y^{i}_{2}}) \rightarrow R_{\mathfrak{I}_{\underline{d}}}$.
Denote by $(\mathscr{E}_{\mathfrak{I}_{\underline{d}}},\mathfrak{N}_{\mathfrak{I}_{\underline{d}}},\phi_{\mathfrak{I}_{\underline{d}}})$ the restriction of $(\mathscr{E}_{Z},\mathfrak{N}_{Z},\phi_{Z})$ to $\mathfrak{I}_{\underline{d}}$. Then we have a universal family of generalized parabolic swamps,
$(\mathscr{E}_{\mathfrak{I}_{\underline{d}}},\underline{q}_{\mathfrak{I}_{\underline{d}}} ,\mathfrak{N}_{\mathfrak{I}_{\underline{d}}},\phi_{\mathfrak{I}_{\underline{d}}})$,
with rank $r$, multidegree $(d_{1},\hdots,d_{l})$ and type $\mathfrak{tp}=(a,b,c,\Lcc,\underline{e})$. Let us denote 
\begin{equation}\label{indicesdegree}
I(r,d,\underline{\kappa},\delta,\mathfrak{tp})=\left\{
(d_{1},\hdots,d_{l})\in\mathbb{N}^{l} \textrm{ }\begin{array}{|l}
\textrm{satisfying the condition $d_1+\hdots+d_l=d$ and}\\
\textrm{such that there exists a } (\underline{\kappa},\delta)\textrm{-semistable swamp} \\
\textrm{with generalized parabolic structure of rank $r$}\\ 
\textrm{multidegree $(d_{1},\hdots,d_{l})$ and type $\mathfrak{tp}$}
\end{array}
\right\}
\end{equation}
From Remark \ref{coletascabron} it follows that for every multiindex  $(d_{1},\hdots,d_{l})\in I(r,d,\underline{\kappa},\delta,\mathfrak{tp})$ we have  $A_-\leq d_i\leq A_{+},i=1,\hdots,l$. Thus, $I(r,d,\underline{\kappa},\delta,\mathfrak{tp})$ is a finite set. Then we define
\begin{equation*}\label{espacioparamgordo}
\mathfrak{I}_{r,d,\mathfrak{tp}}:=\coprod_{\underline{d} \in I(r,d,\delta)}\mathfrak{I}_{\underline{d}}.
\end{equation*}
%%%%%%%%%%%%%%%%%%%%%%%%%%%%%%%%%%%
%%%%%%%%%%%%%%%%%%%%%%%%%%%%%%%%%%%

\subsubsection{The Gieseker space and map}\label{espaciocolega}

%%%%%%%%%%%%%%%%%%%%%%%%%%%%%%%%%%%
%%%%%%%%%%%%%%%%%%%%%%%%%%%%%%%%%%%
We will show that there is a natural closed embedding of the parameter space $\mathfrak{I}_{\underline{d}}$ into certain projective scheme which is $\textrm{SL}(U)$-equivariant.

Fix a Poincare invertible sheaf $\mathcal{P}_{i}$ on $Y_{i}\times \textrm{Pic}^{d_{i}}(Y_{i})$ and let $n\in\mathbb{Z}$. Define the sheaf
%\begin{equation*}
$
\mathcal{G}_{1}^{i}= \mathcal{H}om_{\mathscr{O}_{\textrm{Pic}^{d_{i}}(Y_{i})}}(\bigwedge^{r}U\otimes \mathscr{O}_{\textrm{Pic}^{d_{i}}(Y_{i})},\pi_{\textrm{Pic}^{d_{i}}(Y_{i})*}(\mathcal{P}_{i} \otimes\pi^{*}_{Y_{i}} \mathscr{O}_{Y_{i}}(rn))).
$
%\end{equation*}
The natural number we have fixed satisfies $n>n_0$, therefore the above sheaf is locally free, and we can consider the corresponding projective bundle on $\textrm{Pic}^{d_{i}}(Y_{i})$,
$\mathbb{G}_{1}^{i}=\mathbb{\textbf{P}}(\mathcal{G}_{1}^{i\vee})$.
Note that the determinant map
$\mathscr{E}_{\mathfrak{I}_{\underline{d}}}\mapsto \bigwedge\mathscr{E}_{\mathfrak{I}_{\underline{d}}}|_{Y_{i}}= \bigwedge(\mathscr{E}_{\mathfrak{I}_{\underline{d}}}|_{Y_{i}})$
defines a morphism
$\mathfrak{d}_{i}:\mathfrak{I}_{\underline{d}}\rightarrow \textrm{Pic}^{d_{i}}(Y_{i})$.
Consider now on $\mathfrak{I}_{\underline{d}}\times Y$ the universal quotient $q_{\mathfrak{I}_{\underline{d}}}:U\otimes \pi_{Y}^{*}\mathscr{O}_{Y}(-n) \rightarrow\mathscr{E}_{\mathfrak{I}_{\underline{d}}}$. Restricting to the $i$th component, twisting by $n$ and taking determinants we get
$\bigwedge q_{\mathfrak{I}_{\underline{d}}}^{i}(n): \bigwedge^{r} U\otimes\mathscr{O}_{\mathfrak{I}_{\underline{d}}\times Y_{i}} \rightarrow \bigwedge^{r}\mathscr{E}_{\mathfrak{I}_{\underline{d}}}|_{Y_{i}}\otimes \pi_{Y_{i}}^{*} \mathscr{O}_{Y_{i}}(nr)$.
Let $\mathscr{N}_{i}$ be an invertible locally free sheaf on $\mathfrak{I}_{\underline{d}}$ such that $\bigwedge^{r}\mathcal{E}_{\mathfrak{I}_{\underline{d}}}|_{Y_{i}}= (\mathfrak{d}_{i}\times id_{Y_{i}})^{*}\mathcal{P}_{i}\otimes \pi_{\mathfrak{I}_{\underline{d}}}^{*}\mathscr{N}_{i}$. Then, we have a point 
$\pi_{\mathfrak{I}_{\underline{d}}*}(\bigwedge q_{\mathfrak{I}_{\underline{d}}}^{i}(n))\in \mathbb{G}^{i\bullet}_{1}(\mathfrak{I}_{\underline{d}})$ for each $i$.  

Define now
$
\mathcal{G}_{2}= \mathcal{H}om_{\mathscr{O}_{\textrm{Pic}^{d}(Y)}}(U_{a,b}\otimes\mathscr{O}_{\textrm{Pic}^{d}(Y)},\pi_{\textrm{Pic}^{d}(Y)*}( \mathcal{P}^{\otimes c}\otimes \pi_{Y}^{*}\mathscr{L}\otimes \pi^{*}_{Y}\mathscr{O}_{Y}(na))).
$
%\end{equation*}
For $n>n_0$, $\mathcal{G}_{2}$ is also locally free and we can consider the corresponding projective bundle on $\textrm{Pic}^{\underline{d}}(Y)$,
$\mathbb{G}_{2}=\mathbb{\textbf{P}}(\mathcal{G}_{2}^{\vee})$.
Consider now the universal quotient $\q_{\mathfrak{I}_{\underline{d}}}:U\otimes \mathscr{O}_{\mathfrak{I}_{\underline{d}}\times Y}(-n)\rightarrow \mathscr{E}_{\mathfrak{I}_{\underline{d}}}$
and the universal swamp
$\phi_{\mathfrak{I}_{\underline{d}}}:(\mathscr{E}_{\mathfrak{I}_{\underline{d}}}^{\otimes a})^{\oplus b}\rightarrow \textrm{det}(\mathscr{E}_{\mathfrak{I}_{\underline{d}}})^{\otimes c}\otimes \pi_{Y}^{*}\mathscr{L} \otimes\pi_{\mathfrak{I}_{\underline{d}}}^{*} \mathfrak{N}_{\mathfrak{I}_{\underline{d}}}$.
%%%%%%%%%%
Let $\mathcal{N}$ be an invertible sheaf on $\mathfrak{I}_{\underline{d}}$ such that $\textrm{det}(\mathcal{E}_{\mathfrak{I}_{\underline{d}}})= (\mathfrak{d}\times \textrm{id})^{*}\mathcal{P}\otimes \pi_{\mathfrak{I}_{\underline{d}}}^{*}\mathcal{N}$ and note $U_{a,b}\otimes \mathscr{O}_{\mathfrak{I}_{\underline{d}}\times Y} \simeq \pi_{\mathfrak{I}_{\underline{d}}}^{*}(U_{a,b}\otimes \mathscr{O}_{\mathfrak{I}_{\underline{d}}})$.
Composing  $(\q_{\mathfrak{I}_{\underline{d}}}(n)^{\otimes a})^{\oplus b}$ with the swamp $\phi_{\mathfrak{I}_{\underline{d}}}$, taking $\pi_{\mathfrak{I}_{\underline{d}}*}$ and composing with the adjunction morphism $\psi:U_{a,b}\otimes \mathscr{O}_{\mathfrak{I}_{\underline{d}}}\rightarrow \pi_{\mathfrak{I*_{\underline{d}}}}\pi_{\mathfrak{I}_{\underline{d}}}^{*}(U_{a,b}\otimes\mathscr{O}_{\mathfrak{I}_{\underline{d}}})$ we get a point $\psi\circ(\pi_{\mathfrak{I}_{\underline{d}}*}(\phi_{\mathfrak{I}_{\underline{d}}} \circ (\q_{\mathfrak{I}_{\underline{d}}}(n)^{\otimes a})^{\oplus b}))\in \mathbb{G}_{2}^{\bullet}(\mathfrak{I}_{\underline{d}})$.
%%%%%%%%%%

Altogether, with the obvious morphism to the Grasmannians, $\mathfrak{I}_{\underline{d}}\rightarrow \mathcal{G}r_{1}\times \hdots\times \mathcal{G}r_{\nu}$, give us the so called Gieseker morphism
\begin{equation}\label{giesmap}
\xymatrix{
\textrm{Gies}:\mathfrak{I}_{\underline{d}}\ar[r]  & (\mathbb{G}_{1}^{1} \times\hdots\times\mathbb{G}_{1}^{l}) \times_{\textrm{Pic}^{\underline{d}}(Y)}\mathbb{G}_{2}\times (\mathcal{G}r_{1}\times\hdots\times \mathcal{G}r_{\nu})=: \mathbb{G}.\\
}
\end{equation}
\begin{proposition}\label{hastalaminga100}
The Gieseker morphism $\emph{Gies}:\mathfrak{I}_{\underline{d}}\rightarrow \mathbb{G}$ is injective and $\emph{SL}(U)$-equivariant.
\end{proposition}
\begin{proof}
Follows as in the connected case (see for instance \cite[Lemma 4.3]{gies}).
\end{proof}

%%%%%%%%%%%%%%%%%%%%%%%%%%%%%%%%%%%
%%%%%%%%%%%%%%%%%%%%%%%%%%%%%%%%%%%

\subsection{Semistability}\label{stabilitysec}

We will see that making $n>n_0$ even larger, $\mathfrak{I}_{\underline{d}}$ contains all $(\underline{\kappa},\delta)$-(semi)stable generalized parabolic swamps of fixed type and fixed Hilbert polynomial.
In order to show that the quotient $\mathfrak{I}_{\underline{d}}^{(\underline{\kappa},\delta)\textrm{-(s)s}}\slas\textrm{SL}(U)$ exists and is projective we first find a linearized invertible sheaf on $\mathbb{G}$ for which $\textrm{Gies}^{-1}(\mathbb{G}^{\textrm{(s)s}})=\mathfrak{I}_{\underline{d}}^{(\underline{\kappa},\delta)\textrm{-(s)s}}$ and then we show that $\textrm{Gies}|_{\mathfrak{I}_{\underline{d}}^{(\underline{\kappa},\delta)\textrm{-(s)s}}}$ is a proper morphism. The main auxiliary result is given in Subsection \ref{sectionsectional} (see Theorem \ref{sectionalsemi}) regarding the sectional semistability condition.

\subsubsection{Semistability in the Gieseker space}\label{ovni}

%%%%%%%%%%%%%%%%%%%%%%%%%%%%%%%%%%%
%%%%%%%%%%%%%%%%%%%%%%%%%%%%%%%%%%%
Let $i_1,\hdots,i_{\nu'}$ be the indices in $I(\underline{e})$.
%Without loss of generality we may assume $e_{i}\neq 0$ for each $i=1,\hdots,\nu$.
Let $b_{1},\hdots,b_{l},c,k_{i_1},\hdots,k_{i_{\nu'}}$ be positive integers and consider  the ample invertible sheaf on $\mathbb{G}$,
$\mathscr{O}_{\mathbb{G}}(b_{1},\hdots,b_{l},c,k_{i_1},\hdots,k_{i_{\nu'}})$. %on $\mathbb{G}$.
Consider the obvious linearization on it
and let $\mathbb{G}^{(s)s}$ be the set of points which are (semi)stable with respect to the given linearization. Consider a weighted flag $(U^{\bullet},\underline{m})$, where
$U^{\bullet}:(0)\subset U_{1}\subset\hdots\subset U_{s}\subset U$,
and $\underline{m}=(m_{1},\hdots m_{s})$. Let $\lambda:\GG\rightarrow \textrm{SL}(U)$ be a one parameter subgroup whose weighted flag is $(U^{\bullet},\underline{m})$. Let $t$ be a rational point of $\mathfrak{I}_{\underline{d}}$ and $\textrm{Gies}(t)=(t_{1,1},\hdots,t_{1,l},t_{2},t_{3,1},\hdots,t_{3,\nu})$ its image in $\mathbb{G}$. Let
$q_{t}:U\otimes\mathscr{O}_{Y}(-n)\rightarrow\mathscr{E}$
be the locally free quotient sheaf corresponding to $t$. The weighted filtration $(U^{\bullet},\underline{m})$ induces a filtration of $\mathscr{E}$ defined by $\mathscr{E}_{u}:=q(U_{u}\otimes\mathscr{O}_{Y}(-n))\subset\mathscr{E}$.
 Assume that $h^{1}(Y,\mathscr{E}_{u}(n))=0$ and $l_{u}:=\textrm{dim} (U_{u})=h^{0}(Y,\mathscr{E}_{u}(n))$.
Then, the semistability function is given by (see Section \ref{gitcalculations})
\begin{equation}\label{semisemi}
\begin{split}
\mu(\lambda,\textrm{Gies}(t))&=\sum_{i=1}^{l}b_{i} \mu_{\mathbb{G}_{1}}(\lambda,t_{1,i})+ c\mu_{\mathbb{G}_{2}}(\lambda,t_{2})+ \sum_{i=1}^{\nu}k_{i}\mu_{\mathcal{G}r}(\lambda,t_{3,i})= \\
&=\sum_{i=1}^{l}b_{i}\sum_{u=1}^{s}m_{u} (\textrm{rk}(\mathscr{E}^{i}_{u})p-rh^{0}(Y,\mathscr{E}_{u}(n)))+ \\
&+c\sum_{u=1}^{s}m_{u} (\nu(I_{0},l_{u})p-ah^{0}(Y,\mathscr{E}_{u}(n))) + \\
&+\sum_{i\in I(\underline{e})}k_{i}\sum_{u=1}^{s} m_{u} (p \ \textrm{dim}(q_{i}(\mathscr{E}_{u}(y_{1}^{i})\oplus\mathscr{E}_{u}(y_{2}^{i})))-e_{i}h^{0}(Y,\mathscr{E}_{u}(n))).
\end{split}
\end{equation}
We fix now a concrete polarization, defined as follows ($\textrm{recall }h_{i}=\textrm{deg}(\mathscr{O}_{Y}|_{Y_{i}})$),
\begin{equation}\label{descriptionpolarization}
\left\{\begin{array}{l}
b_{i}:=bh_{i}\textrm{, }b:=p-b', \ b':=b'_{1}+b'_{2}\textrm{, }b'_{1}:=a\delta \textrm{, }b'_{2}:=r\sum_{j\in I(\underline{e})}\kappa_{j}, \\
c:=\delta r h=\sum_{i=1}^{l}\delta r h_{i} \\
k_{i}:=\dfrac{r}{e_i}\kappa_{i}\alpha.
\end{array}\right. 
\end{equation}
Then, Equation (\ref{semisemi}) becomes, 
\begin{align*}
\mu(\lambda,\textrm{Gies}(t))=\sum_{u=1}^{s} m_{u} &\Biggl\{ b\alpha_{u}p-h^{0}(Y,\mathscr{E}_{u}(n))\alpha p+ \Biggr.\\
&+\left. c\nu(I_{0},l_{u})p+\sum_{i\in I(\underline{e})}\alpha\kappa_{i}\dfrac{r}{e_i} p \textrm{dim}(q_{i}(\mathscr{E}_{u}(y_{1}^{i})\oplus\mathscr{E}_{u}(y_{2}^{i}))\right\}.
\end{align*}
Again, since $b=p-b'_{1}-b'_{2}$, $b'_{1}=a\delta$ and $\alpha_{u}=\sum_{i=1}^{l} h_{i}\textrm{rk}(\mathscr{E}^{i}_{u})$, we get
\begin{align*}
\dfrac{\mu(\lambda,\textrm{Gies}(t))}{p}=\sum_{u=1}^{s} m_{u} &\Biggl\{p\alpha_{u}-\alpha h^{0}(Y,\mathscr{E}_{u}(n))+ 
\delta\sum_{i=1}^{l}h_{i}( r\nu(I_{0},l_{u})-a\textrm{rk}(\mathscr{E}_{u}^{i}))+ \Biggr.\\
&\left.+\sum_{i\in I(\underline{e})} \alpha\kappa_{i}\dfrac{r}{e_i}\textrm{dim}(q_{i}(\mathscr{E}_{u}(y_{1}^{i})\oplus\mathscr{E}_{u}(y_{2}^{i}))-b'_{2}\alpha_{u}\right\}.
\end{align*}
Since the first cohomology groups are assumed to be $0$, we find
\begin{equation*}
p\alpha_{u}-\alpha h^{0}(\mathscr{E}_{u}(n))= \alpha_{u}P_{\mathscr{E}}(n)-\alpha P_{\mathscr{E}_{u}}(n)= \alpha_{u}\textrm{deg}(\mathscr{E})-\alpha \textrm{deg}(\mathscr{E}_{u}).
\end{equation*}
We also know that $\underline{\kappa}\textrm{-pardeg}(\mathscr{E}_{u})=\textrm{deg}(\mathscr{E}_{u})-\sum_{i\in I(\underline{e})}\kappa_{i}\dfrac{r}{e_i}\textrm{dim}(q_{i}(\mathscr{E}_{u}(y_{1}^{i})\oplus\mathscr{E}_{u}(y_{2}^{i}))$ and $\underline{\kappa}\textrm{-pardeg}(\mathscr{E})=\textrm{deg}(\mathscr{E})-r(\sum_{i\in I(\underline{e})}\kappa_{i})$.
Then, we finally get
\begin{equation*}\label{stabilitypolarization}
\dfrac{\mu(\lambda,\textrm{Gies}(t))}{p}=\sum_{u=1}^{s} m_{u} \biggl\{(\alpha_{u}\underline{\kappa}\textrm{-pardeg}(\mathscr{E}) - \alpha \underline{\kappa}\textrm{-pardeg}(\mathscr{E}_{u}))+%\Biggr. \\ 
%&\Biggl.
\delta(\alpha\nu(I_{0},l_{u})-a\alpha_{u})\biggr\}.
\end{equation*}

%%%%%%%%%%%%%%%%%%%%%%%%%%%%%%%%%%%
%%%%%%%%%%%%%%%%%%%%%%%%%%%%%%%%%%%

\subsubsection{Sectional semistability}\label{sectionsectional}

%%%%%%%%%%%%%%%%%%%%%%%%%%%%%%%%%%%
%%%%%%%%%%%%%%%%%%%%%%%%%%%%%%%%%%%

Given a swamp with generalized parabolic structure, $(\mathscr{E},\underline{q},\phi)$ rank $r$, degree $d$ and type $\mathfrak{tp}=(a,b,c,\Lcc,\underline{e})$, we will use the following notation,

$$
\left\{\begin{array}{l}
\textrm{par}\chi(\mathscr{E}(n)):=\chi(\mathscr{E}(n))-\sum_{i\in I(\underline{e})} \kappa_{i}\dfrac{r}{e_i} \ \textrm{dim} \ q_{i}(\mathscr{F}(y_{1}^{i})\oplus\mathscr{F}(y_{2}^{i})), \\
\textrm{par}h^{0}(\mathscr{E}(n)):=h^{0}(Y,\mathscr{E}(n))-\sum_{i\in I(\underline{e})} \kappa_{i}\dfrac{r}{e_i} \ \textrm{dim} \ q_{i}(\mathscr{F}(y_{1}^{i})\oplus\mathscr{F}(y_{2}^{i})),\\
\textrm{par}\mu(\mathscr{E}):=\dfrac{\underline{\kappa}\textrm{-pardeg}(\mathscr{E})}{\alpha}.
\end{array}\right.
$$

In the next theorem we adapt the result  \cite[Theorem 2.12]{Alexander3} to our case.

\begin{theorem}\label{sectionalsemi}
There exists $n_{2}\in\mathbb{N}$ such that for very $n>n_{2}$ and every $(\underline{\kappa},\delta)$-(semi)stable generalized parabolic swamp, $(\mathscr{E},\underline{q},\phi)$, the following inequality
\begin{equation*}
\sum_{i=1}^{s}m_{i}(\emph{par}\chi(\mathscr{E}(n))\alpha_{i}-\emph{par}h^{0}(\mathscr{E}_{i}(n))\alpha)+\delta\mu(\mathscr{E}_{\bullet},\underline{m},\phi)(\geq)0
\end{equation*}
holds true for every weighted filtration $(\mathscr{E}_{\bullet},\underline{m})$.
\end{theorem}
\begin{proof}
Let $(\mathscr{E}_{\bullet},\underline{m})$ be a weighted filtration. Assume that each $\mathscr{E}_{i}$ satifies that $\mathscr{E}_{i}(n)$ is globally generated and $h^{1}(Y,\mathscr{E}_{i}(n))=0$ for each $i=1,\hdots,s$. Then, for each $i$ we have
$
\textrm{par}\chi(\mathscr{E}(n))\alpha_{i}-\textrm{par}h^{0}(\mathscr{E}_{i}(n))\alpha=\underline{\kappa}\textrm{-pardeg}(\mathscr{E})\alpha_{i}-\underline{\kappa}\textrm{-pardeg}(\mathscr{E}_{i})\alpha,
$
and we are done. Let $C_{1}$ be the constant given in Proposition \ref{boundparabolicSwamp2} and let $C_{2}$ be another constant.
Consider the bounded family of isomorphism classes of locally free sheaves $\mathscr{E}'$ satisfying
$\textrm{a) }\mu(\mathscr{E}')\geq\frac{d}{\alpha}-C_{2}$,
$\textrm{ b) }1\leq\alpha'\leq\alpha-1$ and
$\textrm{ c) }\mu_{\textrm{max}}(\mathscr{E}')\leq\frac{d}{\alpha}+C_{1}$.
Let $\mathscr{E}$ be a locally free sheaf appearing in a a $(\underline{\kappa},\delta)$-(semi)stable swamp of rank $r$ and degree $d$, and let $\mathscr{E}'\subset\mathscr{E}$ be a locally free subsheaf that do not belongs to the above family. Applying Le Potier-Simpson Estimate to the factors of the Harder-Narashimham filtration of $\mathscr{E}'$ (see \cite[Corollary 3.3.8]{Huyb}), we get
\begin{equation*}
h^{0}(\mathscr{E}'(n))\leq\alpha' (\frac{\alpha'-1}{\alpha'} [\frac{d}{\alpha}+C_{1}+n+B]_{+}+\frac{1}{\alpha'}[\frac{d}{\alpha}-C_{2}+n+B]_{+}),
\end{equation*}
where $B:=-1+\alpha(\alpha+1)/2$. Assume $n$ is large enough so that $\frac{d}{\alpha}+C_{1}+n+B$ and $\frac{d}{\alpha}-C_{2}+n+B$ are positive. Then, $h^{0}(\mathscr{E}'(n))\leq  \alpha'(\frac{d}{\alpha}+n+B-\frac{C_{2}}{\alpha}+C_{1}(\alpha-1))$.
From this we deduce that
%\begin{align*}
$\chi(\mathscr{E}(n))\alpha'-h^{0}(\mathscr{E}'(n))\alpha
\geq -[B']_{+}\alpha^{2}+C_{2} -C_{1}\alpha(\alpha-1)^{2}$,
%\end{align*}
where $B':=B+\dfrac{d}{\alpha}$. Since $B$ depends only on $\alpha$, we can define the constant $K=K(C_{1},C_{2},\alpha,l,\underline{\kappa},d):=-[B']_{+}\alpha^{2}+C_{2} -C_{1}\alpha(\alpha-1)^{2}-r\alpha(\sum_{j\in I(\underline{e})}\kappa_{j})$.
Then,
$\textrm{par}\chi(\mathscr{E}(n))\alpha_{i}-\textrm{par}h^{0}(\mathscr{E}_{i}(n))\alpha
\geq -[B']_{+}\alpha^{2}+C_{2} -C_{1}\alpha(\alpha-1)^{2}-r\alpha(\sum_{j\in I(\underline{e})}\kappa_{j}).$
Let $C_{2}$ be large enough so that $K>\delta a(\alpha-1)$ and let $n$ be large enough so that, for every $\mathscr{E}'$ satisfying a), b) and c), $h^{1}(Y,\mathscr{E}'(n))=0$ and $\mathscr{E}'$ is globally generated. Let $(\mathscr{E}_{\bullet},\underline{m})$ be a weighted filtration with $\mathscr{E}_{\bullet}\equiv (0)\subset \mathscr{E}_{1}\subset\hdots\subset\mathscr{E}_{s}\subset \mathscr{E}$ and $\underline{m}=(m_{1},\hdots,m_{s})$. We make a partition of this filtration as follows. Let $j_{1},\hdots,j_{t}$ be the indices such that  $\mu(\mathscr{E}_{j_{i}})\geq\dfrac{d}{\alpha}-C_{2}$, $\mathscr{E}_{j_{i}}(n) \textrm{ is globally generated }$ and $h^{1}(Y,\mathscr{E}_{j_{i}}(n))=0$
for $i=1,\hdots, t$. Let $l_{1},\hdots,l_{s-t}$ the set of indices $\{1,2,\hdots,s\}\setminus\{j_{1},\hdots,j_{t}\}$ in increasing order. Define the weighted filtrations $(\mathscr{E}_{1,\bullet},\underline{m}_{1})$ and $(\mathscr{E}_{2,\bullet},\underline{m}_{2})$ as
\begin{align*}
\mathscr{E}_{\bullet,1}&\equiv \ (0)\subset\mathscr{E}_{j_{1}}\subset\hdots\subset\mathscr{E}_{j_{t}}\subset\mathscr{E}, \ \ \underline{m}_{1}=(m_{j_{1}},\hdots,m_{j_{t}}), \\
\mathscr{E}_{\bullet,2}&\equiv \ (0)\subset\mathscr{E}_{l_{1}}\subset\hdots\subset\mathscr{E}_{l_{s-t}}\subset\mathscr{E}, \ \underline{m}_{2}=(m_{l_{1}},\hdots,m_{l_{s-t}}).
\end{align*}
From Equation (\ref{lema bound}) we find that 
$\mu(\mathscr{E}_{\bullet},\underline{m},\phi)\geq \mu(\mathscr{E}_{\bullet,2},\underline{m}_{2},\phi)-(\sum_{q=1}^{t}m_{j_{q}})a(\alpha-1).$
Thus
\begin{align*}
&\sum_{i=1}^{s}m_{i}(\textrm{par}\chi(\mathscr{E}(n))\alpha_{i}-\textrm{par}h^{0}(\mathscr{E}_{i}(n))\alpha)+\delta\mu(\mathscr{E}_{\bullet},\underline{m},\phi)\geq \\
\geq& \sum_{q=1}^{t}m_{j_{q}}(\textrm{par}\chi(\mathscr{E}(n))\alpha_{j_{q}}-\textrm{par}h^{0}(\mathscr{E}_{j_{q}}(n))\alpha)+\delta\mu(\mathscr{E}_{\bullet,1},\underline{m}_{1},\phi) + \\
&+(\sum_{q=1}^{s-t}m_{l_{q}})K-\delta(\sum_{q=1}^{s-t}m_{l_q})a(\alpha-1)\geq 0,
\end{align*}
and the result is proved.
\end{proof}

\subsubsection{$(\kappa,\delta)$-semistability and Hilbert-Mumford semistability}

The goal now is to prove Theorem \ref{melon3}, which shows that $(\underline{\kappa},\delta)$-(semi)stability is equivalent to GIT (semi) stability in the Gieseker space under some conditions.

Let $B:=-1+\alpha(\alpha+1)/2$ be the constant given in the proof of Theorem \ref{sectionalsemi} and let $K'$ be a constant such that $d+K'>0$ and with the property
\begin{equation}\label{konstante}
\alpha K'>\textrm{max}\biggl\{d(w-\alpha)+\alpha r\nu+a\delta(\alpha-1)
+B\alpha(\alpha-1)|w=1\hdots\alpha-1\biggr\},
\end{equation}
\begin{proposition}\label{lemacarra2}
There exists $n_{3}\in\mathbb{N}$ and a constant $C_{3}$ such that for every $n\geq n_{3}$ and for any triple $t=(q:U\otimes\mathscr{O}_{Y}(-n)\rightarrow \mathscr{E},\underline{q},\phi)$ of degree $d$ and multiplicity $\alpha$ whose induced map $U\rightarrow H^{0}(Y,\mathscr{E}(n))$ is injective and giving a semistable point in the Gieseker space, $\mathbb{G}^{(s)s}$,
$\mu_{\emph{max}}(\mathscr{E})\leq\mu(\mathscr{E})+C_{3}$.
\end{proposition}
\begin{proof} 
It is enough to show that $\textrm{deg}(\mathscr{E}')<d+K'$ for the maximal destabilizing subsheaf, since in such case we would have
$\mu(\mathscr{E}'')\leq\mu(\mathscr{E}')<\dfrac{d+K'}{\alpha(\mathscr{E}')}\leq d+K'\leq \mu(\mathscr{E})+C_3$ for every subsheaf $\mathscr{E}''\subset \mathscr{E}$, where $C_3:=\mu(\mathscr{E})(\alpha-1)+K'$.

Let $\mathscr{Q}:=\mathscr{E}/\mathscr{E}'$ be the (semistable) quotient locally free sheaf. Let us use the notation $\alpha':=\alpha(\mathscr{E}'),\alpha'':=\alpha(\mathscr{Q}), d':=\textrm{deg}(\mathscr{E}'), d'':=\textrm{deg}(\mathscr{Q}), \mu':=\mu(\mathscr{E}')$ and $\mu'':=\mu(\mathscr{Q})$. Assume that $d'\geq d+K'$ and and let us show that we get a contradiction. For all $n\in\mathbb{N}$ we have
$h^{0}(Y,\mathscr{Q}(n))\leq\alpha''[\mu''+n+B]_{+}$.
Then we have to study two different cases.\\
Consider the first case, $h^{0}(Y,\mathscr{Q}(n))\leq\alpha''(\mu''+n+B)$. Set $U':=H^{0}(Y,\mathscr{E}'(n))\cap U$. Then we have,
\begin{equation*}
\begin{split}
\textrm{dim}(U')&\geq p-h^{0}(Y,\mathscr{Q}(n))\geq 
\alpha(\dfrac{1-g}{h})+d+\alpha n-\alpha''(\mu''+n+B)\geq \\
&\geq\alpha(\dfrac{1-g}{h}+n)+d-d''-\alpha''(\dfrac{1-g}{h}+n)-\alpha''B\geq\\
&\geq \alpha'(\dfrac{1-g}{h}+n)+d+K'-B(\alpha-1).
\end{split}
\end{equation*}
Consider the locally free sheaf $\widehat{\mathscr{E}}:=\textrm{Im}(U'\otimes\mathscr{O}_{Y}(-n)\rightarrow\mathscr{E}_{t})$.
Thus, we have $U'\subset H^{0}(Y,\widehat{\mathscr{E}}(n))\cap U$ (see \cite[Lemma 3.3 ]{sols}, which also holds true in our case), $\textrm{rk}(\widehat{\mathscr{E}}|_{Y_{i}})\leq \textrm{rk}(\mathscr{E}'|_{Y_{i}})$ and $\widehat{\mathscr{E}}$ is generically generated by global sections. Let $\{u_{1},\hdots,u_{i}\}$ be a basis for $U'$ and complete it to a basis $\underline{u}=\{u_{1},\hdots,u_{p}\}$ of $U$. Let $\lambda=\lambda(\underline{u},\gamma_{p}^{(i)})$ be the associated one parameter subgroup. Then we clearly have that 
$\mu_{\mathbb{G}_{1}^{i}}(\lambda,i_{1,i}(t))=p \textrm{rk}(\widehat{\mathscr{E}}|_{Y_{i}})-r\textrm{dim}(U')
\leq p \textrm{rk}(\mathscr{E}'|_{Y_{i}})-r \textrm{dim}(U').$
Since $\nu(I,i)\leq a$, we also have
$\mu_{\mathbb{G}_{2}}(\lambda,i_{2}(t))\leq a(p-\textrm{dim}(U'))$.
Therefore,
\begin{align*}
\mu_{\mathbb{G}}(\lambda,\textrm{Gies}(t))=& \sum_{i=1}^{l}b_{i}\mu_{\mathbb{G}_{1}^{i}} (\lambda,i_{1,i}(t))+ c\mu_{\mathbb{G}_{2}}(\lambda,i_{2}(t))+\\
&+\sum_{i=1}^{\nu}k_{i}(p \textrm{dim}(q_{i}(\widehat{\mathscr{E}}(y^{i}_{1})\oplus\widehat{\mathscr{E}}(y^{i}_{2})))-e_{i} \textrm{dim}(U'))\leq\\
\leq
&\sum_{i=1}^{l}d_{i}(p-a\delta-r(\sum_{i\in I(\underline{e})}\kappa_{i}))(p \textrm{rk}(\mathscr{E}'|_{Y_{i}}) - r \textrm{dim}(U'))+\\
&+\sum_{i=1}^{l}d_{i}\delta ra(p-\textrm{dim}(U')) +\\
&+\sum_{i\in I(\underline{e})}\kappa_{i}\dfrac{r}{e_i}\alpha(p \textrm{dim}(q_{i}(\mathscr{E}'(y^{i}_{1})\oplus\mathscr{E}'(y^{i}_{2}))))-r h\textrm{dim}(U')).
\end{align*}

An easy calculation give us
\begin{equation}\label{menudotruno}
\begin{split}
\dfrac{\mu_{\mathbb{G}}(\lambda,\textrm{Gies}(t))}{p}\leq&\alpha'(p-r(\sum_{i\in I(\underline{e})}\kappa_{i}))-\alpha\{\textrm{dim}(U')- \\ 
&-\sum_{i\in I(\underline{e})}\kappa_{i}\dfrac{r}{e_i}\alpha(\textrm{dim}(q_{i}(\mathscr{E}'(y^{i}_{1}) \oplus\mathscr{E}'(y^{i}_{2}))))\}+a\delta(\alpha- \alpha').
\end{split}
\end{equation}
Since $p=\alpha(n+\frac{1-g}{h})+d$ and $\textrm{dim}(U')\geq d+K'+\alpha' (n+\frac{1+g}{h})-B(\alpha-1)$,
we deduce that,
\begin{align*}
\dfrac{\mu_{\mathbb{G}}(\lambda,\textrm{Gies}(t))}{p}\leq &a\delta(\alpha-\alpha')-\alpha K'+B\alpha(\alpha-1) -r\alpha'(\sum_{i\in I(\underline{e})}\kappa_{i})+\\
&+ \alpha(\sum_{i\in I(\underline{e})}\kappa_{i}\dfrac{r}{e_i} \textrm{dim}(q_{i}(\mathscr{E}'(y^{i}_{1}) \oplus\mathscr{E}'(y^{i}_{2}))))+d(\alpha'-\alpha).
\end{align*}
Since $\alpha'r(\sum_{i\in I(\underline{e})}\kappa_{i})>0$, 
$\alpha\sum_{i\in I(\underline{e})}\kappa_{i} \dfrac{r}{e_i}\textrm{dim}(q_{i}(\mathscr{E}'(y^{i}_{1}) \oplus\mathscr{E}'(y^{i}_{2})))< \alpha\nu r$ (because $\kappa_{i}<\dfrac{e_i}{r}$) and
$\alpha-\alpha'<\alpha-1$,
we get
$\mu_{\mathbb{G}}(\lambda,\textrm{Gies}(t))<0$.
However $\textrm{Gies}(t)$ is semistable so we get a contradiction.\\
Consider now the second case, $h^{0}(Y,\mathscr{Q}(n))=0$. Assuming $n>\dfrac{g-1}{h}$, we have $\textrm{dim}(U')=p$.
The same calculation as before (see Equation (\ref{menudotruno})) shows that
\begin{align*}
\dfrac{\mu_{\mathbb{G}}(\lambda,\textrm{Gies}(t))}{p}\leq& \ \alpha'(p-r(\sum_{i\in I(\underline{e})}\kappa_{i}))-\alpha\{\textrm{dim}(U')- \\ 
&-\sum_{i\in I(\underline{e})}\kappa_{i}\dfrac{r}{e_i}\textrm{dim}(q_{i}(\mathscr{E}'(y^{i}_{1}) \oplus\mathscr{E}'(y^{i}_{2})))\}+a\delta(\alpha- \alpha')\leq \\
\leq& \ (\alpha'-\alpha)(p-a\delta)+\alpha\nu r.
\end{align*}
Assume $n$ is large enough so that $p-a\delta>\dfrac{-\alpha\nu r}{\alpha'-\alpha}$ (recall that $p=r\chi(\mathscr{O}_{Y})+ d+\alpha n$). Then, $\mu_{\mathbb{G}}(\lambda,\textrm{Gies}(t))< 0$ and we get again a contradiction.
\end{proof}
\begin{theorem}\label{melon3}
There exists $n_{4}\in\mathbb{N}$ such that for every $n\geq n_{4}$, $(\mathscr{E}_{t},\underline{q}_{t},\tau_{t})$ is $(\underline{\kappa},\delta)$-(semi) stable if and only if $t\in \emph{Gies}^{-1}(\mathbb{G}^{(s)s})$.
\end{theorem}

\begin{proof}
1) From the construction of the parameter space, we know that $q_{t}$ induces an isomorphism $U\simeq H^{0}(Y,\mathscr{E}_{t}(n))$.
Then, by Proposition \ref{lemacarra2}, $\textrm{Gies}(t)\in\mathbb{G}^{\delta-(s)s}$ implies $\mu_{\textrm{max}}(\mathscr{E}_{t})\leq\frac{\textrm{deg}(\mathscr{E})}{\alpha}+C_{3}$. We also know, by Proposition \ref{lemacarra4}, that $(\mathscr{E}_{t},\underline{q}_{t},\phi_{t})$ is $(\underline{\kappa,}\delta)$-(semi)stable if and only if
$P_{\underline{\kappa}}(\mathscr{E}_{\bullet},\underline{m})+ \delta\mu(\mathscr{E}_{\bullet},\underline{m},\phi)(\geq) 0$
for every $(\mathscr{E}_{\bullet},\underline{m})$ with $\textrm{par}\mu(\mathscr{E}_{j})\geq \textrm{par}\mu(\mathscr{E})-C_{1}$. Observe that, in this case,
$\mu(\mathscr{E}_{j})>\textrm{par}\mu(\mathscr{E}_{j})\geq \textrm{par}\mu(\mathscr{E})-C_{1}\geq \mu(\mathscr{E})-\frac{\nu}{h}-C_{1}$.
Denote $\overline{C}_{1}=\frac{\nu}{h}+C_{1}$. Consider the family of locally free sheaves satisfying
$\textrm{a) }\mu_{\textrm{max}}(\mathscr{E}')\leq\frac{\textrm{deg}(\mathscr{E})}{\alpha}+C_{3},
\textrm{b) }\textrm{par}\mu(\mathscr{E}')\geq \textrm{par}\mu(\mathscr{E})-C_{1}$ and
$\textrm{ c) }1\leq\alpha'\leq\alpha-1.$
This family is clearly bounded.
Therefore, there is a  natural number, $n\in\mathbb{N}$, large enough such that $\mathscr{E}'(n)$ is globally generated and $h^{1}(Y,\mathscr{E}'(n))=0$ for any $\mathscr{E}'$ of this family. 

Now, fix a weighted filtration $(\mathscr{E}_{\bullet},\underline{m})$ of $\mathscr{E}_{t}$ satisfying conditions a), b) and c). Let $\underline{u}=\{u_{1},\hdots,u_{p}\}$ be a basis of $U$, such that there are indices $l_{1},\hdots,l_{s}$ with $U^{(l_{j})}:=\langle u_{1},\hdots,u_{l_{j}}\rangle\simeq H^{0}(Y,\mathscr{E}_{j}(n))$ for each $j$. Define $\underline{\gamma}=\sum_{j=1}^{s}\alpha_{j}\gamma_{p} ^{(l_{j})}$ and consider the one parameter subgroup, $\lambda(\underline{u},\gamma_{p}^{(l_{j})})$ .
Let $I_{0}$ be a multiindex giving the minimum in $\mu_{\mathbb{G}_{2}}(\lambda(\underline{u},\underline{\gamma}))$. Then $\mu_{\mathbb{G}}(\lambda(\underline{u},\gamma),\textrm{Gies}(t))(\geq 0)$ if and only if $\mu_{\mathbb{G}}(\lambda(\underline{u},\gamma),\textrm{Gies}(t))/p \ (\geq 0)$. But looking at the calculations at the beginning of Section \ref{ovni}, we have
\begin{align*}
0&(\leq)\dfrac{\mu_{\mathbb{G}}(\lambda(\underline{u},\gamma),\textrm{Gies}(t))}{p}=\\
&=\sum_{u=1}^{s} m_{u}\lbrace (\widehat{\alpha}_{u}\underline{\kappa}\textrm{-pardeg}(\mathscr{E}) - \alpha \underline{\kappa}\textrm{-pardeg}(\widehat{\mathscr{E}}_{u}))+  \delta(\alpha\nu(I_{0},l_{u})-a\widehat{\alpha}_{u})\rbrace,
\end{align*}
$\widehat{\mathscr{E}_{i}}$ being the saturated subsheaf generated by $\mathscr{E}_{i}$. Finally, since $\widehat{\alpha_{i}}:=\alpha(\widehat{\mathscr{E}_{i}})=\alpha_{i}$ and $\underline{\kappa}\textrm{-pardeg}(\widehat{\mathscr{E}}_{i})\geq \underline{\kappa}\textrm{-pardeg}(\mathscr{E}_{i})$, we get
\begin{align*}
0(\leq&)\dfrac{\mu_{\mathbb{G}}(\lambda(\underline{u},\gamma),\textrm{Gies}(t))}{p}=\\
=&\sum_{u=1}^{s} m_{u}\lbrace (\alpha{\mathscr{E}}_{u}\underline{\kappa}\textrm{-pardeg}(\mathscr{E}) - \alpha \underline{\kappa}\textrm{-pardeg}(\widehat{\mathscr{E}}_{u}))+  \delta(\alpha\nu(I_{0},l_{u})-a\widehat{\alpha}_{u})\rbrace \leq \\
\leq&\sum_{u=1}^{s} m_{u}\lbrace (\alpha_{u}\underline{\kappa}\textrm{-pardeg}(\mathscr{E}) - \alpha \underline{\kappa}\textrm{-pardeg}(\mathscr{E}_{u}))+  \delta(\alpha\nu(I_{0},l_{u})-a\alpha_{u})\rbrace= \\
=& P_{\underline{\kappa}}(\mathscr{E}_{\bullet},\underline{m})+\delta\mu(\mathscr{E}_{\bullet},\underline{m},\tau).
\end{align*}
Thus, the swamp is $(\underline{\kappa},\delta)$-semistable.

2) By Theorem \ref{sectionalsemi} we deduce that
\begin{equation}\label{oleole}
\sum_{i=1}^{s}m_{i}(\textrm{par}\chi(\mathscr{E}(n))\alpha_{i}-\textrm{par}h^{0}(\mathscr{E}_{i}(n))\alpha)+\delta\mu(\mathscr{E}_{\bullet},\underline{m},\phi)(\geq)0
\end{equation}
for any weighted filtration $(\mathscr{E}_{\bullet},\underline{m})$ of $\mathscr{E}_{t}$. Let $\lambda$ be a one parameter subgroup and $(U_{\bullet},\underline{m}')$ a weighted filtration such that $\lambda=\lambda(U_{\bullet},\underline{m}')$.
This filtration together with the quotient $q_{t}:U\otimes\mathscr{O}_{Y_{t}}(-n)\rightarrow\mathscr{E}_{t}$ induces a chain
\begin{equation}\label{chano}
(0)\subseteq\mathscr{E}'_{1}\subseteq\hdots\subseteq\mathscr{E}'_{s'}\subseteq\mathscr{E}_{t}
\end{equation}
and, therefore, a filtration
$\mathscr{E}_{\bullet}\equiv (0)\subset\mathscr{E}_{1}\subset\hdots\subset\mathscr{E}_{s}\subset\mathscr{E}_{t}$,
formed by the different subsheaves collected in the above chain.
Let $J=(i_{1},\hdots,i_{s})$ be the multiindex defined by the following condition: $i_{j}\in\{1,\hdots,s'\}$ is the maximum index among those $k\in\{1,\hdots,s'\}$ such that $\mathscr{E}_{j}=\mathscr{E}'_{k}$.  Denote by $m_{j}$ the sum of the numbers $m'_{k}$ corresponding to those sheaves in the chain (\ref{chano}) which are equal to $\mathscr{E}_{i}$, i.e.,
$
m_{j}=m_{k}+m_{k+1}+\hdots+m_{i_{j}},
$
$(k,k+1,\hdots,i_{j})$ being the indices such that $\mathscr{E}'_{k}=\mathscr{E}'_{k+1}=\hdots=\mathscr{E}'_{i_{j}}=\mathscr{E}_{j}$. We get in this way a weighted filtration $(\mathscr{E}_{\bullet},\underline{m})$. Multiplying by $p$ in Equation (\ref{oleole}) we get
\begin{align*}
0\leq\sum_{i=1}^{s}&m_{i}\bigg\{p^{2}\alpha_{i}-ph^{0}(Y,\mathscr{E}_{i}(n))\alpha+\delta p(\alpha\nu_{i}(I_{0})-a\alpha_{i})+ \\
+&p\sum_{j\in I(\underline{e})}\kappa_{j}\dfrac{r}{e_j}\textrm{dim}(q_{j}(\mathscr{E}_{i}(y^{j}_{1})\oplus\mathscr{E}_{i}(y_{2}^{j})))\alpha-rp(\sum_{j\in I(\underline{e})}\kappa_{j})\alpha_{i}\bigg\}.
\end{align*}
The inverse calculation presented in Subsection \ref{ovni} gives
\begin{equation}\label{pandereta}
\begin{split}
0\leq &\sum_{u=1}^{l}b_{u}\sum_{i=1}^{s}m_{i} (\textrm{rk}(\mathscr{E}^{u}_{i})p-rh^{0}(Y,\mathscr{E}_{i}(n)))+ \\
&+c\sum_{i=1}^{s}m_{i} (\nu_{i}(I_{0})p-ah^{0}(Y,\mathscr{E}_{i}(n))) + \\
&+\sum_{j\in I(\underline{e})}k_{j}\sum_{i=1}^{s} m_{i} (p \textrm{dim}(q_{j}(\mathscr{E}_{i}(y_{1}^{j})\oplus\mathscr{E}_{i}(y_{2}^{j})))-e_{j}h^{0}(Y,\mathscr{E}_{i}(n))).
\end{split}
\end{equation}
Since $l_{i}:=\textrm{dim} U_{i}\leq h^{0}(Y,\mathscr{E}_{i}(n))$, Equation (\ref{pandereta}) turns into
\begin{equation*}
\begin{split}
0\leq &\sum_{u=1}^{l}b_{u}\sum_{i=1}^{s'}m'_{i} (\textrm{rk}(\mathscr{E}^{u}_{i})p-rl_{i})+ \\
&+c\sum_{i=1}^{s'}m'_{i} (\nu_{i}(I_{0})p-al_{i}) + \\
&+\sum_{j\in I(\underline{e})}k_{j}\sum_{i=1}^{s'} m'_{i} (p \textrm{dim}(q_{j}(\mathscr{E}_{i}(y_{1}^{j})\oplus\mathscr{E}_{i}(y_{2}^{j})))-e_{j}l_{i})=\\
=&\mu_{\mathbb{G}}(\lambda(U_{\bullet},\underline{m}'),\textrm{Gies}(t)),
\end{split}
\end{equation*}
and the proposition is proved.
\end{proof}

%%%%%%%%%%%%%%%%%%%%%%%%%%%%%%%%%%%
%%%%%%%%%%%%%%%%%%%%%%%%%%%%%%%%%%%

\subsection{The moduli space}

%%%%%%%%%%%%%%%%%%%%%%%%%%%%%%%%%%%
%%%%%%%%%%%%%%%%%%%%%%%%%%%%%%%%%%%
The last step before proving the existence of the moduli space consists in showing that the restriction of the Gieseker map to the $(\underline{\kappa},\delta)$-semiststable locus is proper.

\begin{proposition}\label{propergieseker} 
There exists $n$ large enough such that the Gieseker morphism,
$\emph{Gies}:\mathfrak{I}_{\underline{d}}^{(\underline{\kappa},\delta)\emph{-(s)s}}\rightarrow\mathbb{G}^{\emph{s(s)}}$ ,
is proper for any $\underline{d}\in I_{r,d,\delta}$. 
\end{proposition}
%%%%%%%%%%%%%%%%%%%%%
\begin{proof}
For the sake of notation we drop the subindex $\underline{d}$. We use the the valuative criterion for properness. Let $(\mathscr{O},\mathfrak{m},k)$ be a DVR, $K$ being its field of fractions and assume we have a conmutative diagram
$$
\xymatrix{
&\textrm{Spec}(K)\ar[r]^{h_{K}}\ar[d] &\mathfrak{I}_{\underline{d}}^{(\underline{\kappa},\delta)\textrm{-(s)s}}\ar[d] \\
\{0,\eta\}=S:\ar@{=}[r]&\textrm{Spec}(\mathscr{O})\ar[r]^{h} & \mathbb{G}^{(s)s}.
}
$$
The morphism $h_{K}$ is given by a family $(q_{K},\underline{q}_{K},\phi_{K})$ over $Y_{K}:=Y\times \textrm{Spec}(K)$, where
\begin{equation}\label{familyK}
\begin{split}
q_{K}:&U\otimes\mathscr{O}_{Y_{K}}(-n)\twoheadrightarrow\mathscr{E}_{K}\\
\phi_{K}:& (\mathscr{E}_{K}^{\otimes a})^{\oplus b}\rightarrow \textrm{det}(\mathscr{E}_{K})^{\otimes c}\otimes \mathscr{L}_{K}\\
q_{iK}:&\Gamma(\mathscr{E}_{K}|_{y^{i}_{1},y^{i}_{2}})\rightarrow R_{K}
\end{split}
\end{equation}

Let us see that $h_{K}$ can be extended to a family, $\widehat{h}=(q_{S},\phi_{S},\underline{q}_{S})$, over $Y\times S$. The quotient $q_{K}$ defines a point in the Quot scheme of quotients of $U\otimes\mathscr{O}_{Y}(-n)$ with the fixed Hilbert polynomial $P(n)$. Therefore, there exists a (unique) flat extension 
\begin{equation}\label{quotient}
q_{S}:U\otimes\pi^{*}\mathscr{O}_{Y}(-n)\twoheadrightarrow \mathscr{E}_{S}
\end{equation}
over $Y\times S$. Define now the sheaves $\mathscr{M}:=\pi_{S*}(\textrm{det}(\mathscr{E}_{S})^{\otimes c}\otimes\pi_{Y}^{*}\Lcc\otimes\pi_{Y}^{*}\mathscr{O}_{Y}(an))$ and $\mathscr{G}=\pi_{S*}((U^{\otimes a})^{\oplus b}\otimes\pi_{Y}^{*}\mathscr{O}_{Y})$.
Both sheaves are locally free, so we can form the projective space over $S$,
$
\textrm{pr}_{S}:\mathbb{P}:=\mathbf{P}(\textrm{Hom}_{\OO}(\mathscr{G},\Hh)^{\vee})\rightarrow S,
$
which carries a tautological morphism over $\mathbb{P}\times Y$,
$$
\textrm{pr}_{\mathbb{P}}^{*}\textrm{pr}_{\mathbb{P}*}((U^{\otimes a})^{\oplus b}\otimes\textrm{pr}_{Y}^{*}\mathscr{O}_{Y})\rightarrow(\textrm{id}_{Y}\times \textrm{pr}_{S})^{*}\textrm{det}(\mathscr{E}_{S})^{\otimes c}\otimes \textrm{pr}_{Y}^{*}\mathscr{O}_{Y}(an)\otimes\textrm{pr}_{Y}^{*}\Lcc\otimes \textrm{pr}_{\mathbb{P}}^{*}\mathscr{O}_{\mathbb{P}}(1)
$$
Now, the canonical morphism $\Delta:\textrm{pr}_{\mathbb{P}}^{*}\textrm{pr}_{\mathbb{P}*}((U^{\otimes a})^{\oplus b}\otimes\pi_{Y}^{*}\mathscr{O}_{Y})\rightarrow (U^{\otimes a})^{\oplus b}\otimes\pi_{Y}^{*}\mathscr{O}_{Y}$  induces a diagram
$$
\xymatrix{
K\ar[r] \ar@{..>}[rd]_{g}& \textrm{pr}_{\mathbb{P}}^{*}\textrm{pr}_{\mathbb{P}*}((U^{\otimes a})^{\oplus b}\otimes\textrm{pr}_{Y}^{*}\mathscr{O}_{Y})\ar[r]\ar[d] & (\textrm{id}_{Y}\times\textrm{pr}_{S})^{*}(\mathscr{E}_{S}(n)^{\otimes a})^{\oplus b}\\
&\mathcal{H}'& ,
}
$$
where $\mathcal{H}'=(\textrm{id}_{Y}\times \textrm{pr}_{S})^{*}\textrm{det}(\mathscr{E}_{S})^{\otimes c}\otimes \textrm{pr}_{Y}^{*}\mathscr{O}_{Y}(an)\otimes\textrm{pr}_{Y}\Lcc\otimes \textrm{pr}_{\mathbb{P}}^{*}\mathscr{O}_{\mathbb{P}}(1)$.
Let $\mathbb{S}\subset\mathbb{P}$ be the closed subscheme over which $g$ is the zero morphism, i.e., over which the tautological morphism factorizes through $(\textrm{id}_{Y}\times\textrm{pr}_{S})^{*}(\mathscr{E}_{S}(n)^{\otimes a})^{\oplus b}$. Thus, we have over $\mathbb{S}\times Y$ a morphism
$
(\textrm{id}_{Y}\times\textrm{pr}_{S})^{*}(\mathscr{E}_{S}^{\otimes a})^{\oplus b}\rightarrow (\textrm{id}_{Y}\times \textrm{pr}_{S})^{*}\textrm{det}(\mathscr{E}_{S})^{\otimes c}\otimes\textrm{pr}_{Y}\Lcc\otimes \textrm{pr}_{\mathbb{P}}^{*}\mathscr{O}_{\mathbb{P}}(1).
$ 
Note now that the morphism $\phi_{K}:(\mathscr{E}_{K}^{\otimes a})^{\oplus b}\rightarrow \textrm{det}(\mathscr{E}_{K})^{\otimes c}\otimes \mathscr{L}_{K}$ defines a point $\textrm{Spec}(K)\rightarrow\mathbb{S}$. Since $\mathbb{S}$ is projective this point extends (uniquely) to a point $\textrm{Spec}(\OO)\rightarrow\mathbb{S}$, i.e., to a morphism 
\begin{equation}\label{Swamp}
\phi_{S}:(\mathscr{E}_{S}^{\otimes a})^{\otimes b}\rightarrow \textrm{det}(\mathscr{E}_{S})^{\otimes c}\otimes \pi_{Y}^{*}\Lcc\otimes \mathcal{N}
\end{equation}
Let us extend now the parabolic structure. Since $\mathscr{E}_{S,\eta}\simeq \mathscr{E}_{K}$ we have an isomorphism $\pi_{K*}(\mathscr{E}_{S,\eta}|_{D_{i}})\simeq \pi_{K*}(\mathscr{E}_{K}|_{D_{i}})$. Thus composing with $\pi_{K*}(\mathscr{E}_{K}|_{D_{i}})\twoheadrightarrow R_{K}$, we get a surjection
$
\pi_{K*}(\mathscr{E}_{S,\eta}|_{D_{i}})\twoheadrightarrow R_{K}.
$
Observe that the morphism $\pi_{S}:D_{i}\times S\rightarrow S$ is finite, thus affine and proper. By flat base change, we know that
$
\pi_{K*}(\mathscr{E}_{S,\eta}|_{D_{i}})=j^{*}\pi_{S*}(\mathscr{E}_{S}|_{D_{i}}),
$
$j$ being the open embedding $j:\eta\hookrightarrow S$. Now, taking the push-forward and composing with the canonical map $\pi_{S*}(\mathscr{E}_{S}|_{D_{i}})\rightarrow  j_{*}j^{*}\pi_{S*}(\mathscr{E}_{S}|_{D_{i}})$, we get a morphism $\pi_{S*}(\mathscr{E}_{S}|_{D_{i}}) \rightarrow j_{*}R_{K}$.
Let $R_{S}\subset  j_{*}R_{K}$ be its image. Then by \cite[Proposition 2.8.1]{Groth1}, $R_{S}$ is $S$-flat (thus a free $\OO$-module) and the quotient 
\begin{equation}\label{parabolic}
q_{iS}:\pi_{S*}(\mathscr{E}_{S}|_{D_{i}})\twoheadrightarrow R_{S} 
\end{equation}
extends $q_{iS}:\pi_{K*}(\mathscr{E}_{K}|_{D_{i}})\twoheadrightarrow R_{K}$ (thus $\textrm{rk}(R_{S})=e_{i}$). Then the family given in Equations (\ref{quotient}), (\ref{Swamp}), (\ref{parabolic}), $\widehat{h}=(q_{S},\phi_{S},\underline{q}_{S})$, extends the family given in Equation (\ref{familyK}) to $S$. Clearly, the family $(q_{S},\phi_{S},\underline{q}_{S})$ defines an $S$-valued point $t:S\rightarrow\mathbb{G}$ in the Gieseker space. 
Since $t(\eta)=h(\eta)$ we deduce that $t(0)=h(0)$, thus it defines a semistable point in $\mathbb{G}$. Let us show that $q_{(0)}$ induces an isomorphism $U\simeq H^{0}(Y,\mathscr{E}_{(0)}(n))$. To show that it is injective, we consider the kernel, $H\subset U$, of $H^{0}(q_{(0)}(n)):U\rightarrow H^{0}(Y,\mathscr{E}_{(0)}(n))$. Since $t(0)$ is semistable we have,
\begin{align*}
\mu_{\mathbb{G}}(\lambda,t(0))&= \sum_{i=1}^{l}b_{i}\mu_{\mathbb{G}_{1}^{i}} (\lambda,t_{1,i}(0))+ c\mu_{\mathbb{G}_{2}}(\lambda,t_{2}(0))+\\
&+\sum_{i\in I(\underline{e})}k_{i}\mu_{\mathcal{G}r}(\lambda,t_{3,i}(0))=\\
&=\sum_{i=1}^{l}b_{i}(-r \textrm{dim}(H))+ca(-\textrm{dim}(H))+\\
&+\sum_{i\in I(\underline{e})}k_{i}(p \textrm{dim}(t_{i0}(H\oplus H)-e_{i} \textrm{dim}(H))=\\
&=\sum_{i=1}^{l}d_{i}(p-a\delta-r\sum_{j\in I(\underline{e})}\kappa_{j})( - r \textrm{dim}(H))+\sum_{i=1}^{l}d_{i}\delta ra(-\textrm{dim}(H)) +\\
&+\sum_{i\in I(\underline{e})}\kappa_{i}\alpha(-r \textrm{dim}(H))=-\alpha p \textrm{dim}(H)\geq0
\end{align*}
so we must have $\textrm{dim}(H)=0$, i.e, $U\rightarrow H^{0}(Y,\mathscr{E}_{(0)}(n))$ is injective. Let us show that it is in fact an isomorphism. For that we just need to show that $h^{1}(Y,\mathscr{E}_{(0)}(n))=0$. Suppose it does not. Then, by Serre duality, there is a non trivial morphism $\mathscr{E}_{(0)}(n)\rightarrow \omega_{Y}$. Let $\mathscr{G}$ be its image, and consider the linear map
$
\Omega:U\hookrightarrow H^{0}(Y,\mathscr{E}_{(0)}(n))\rightarrow H^{0}(Y,\mathscr{G}).
$
Let $H\subset U$ be the kernel of $\Omega$, let $\lambda$ be the corresponding one parameter subgroup and $\mathscr{F}\subset \mathscr{E}_{(0)}$ the subsheaf generated by $H$. Since $t(0)$ is semistable, we get:
\begin{align*}
0\leq\dfrac{\mu(\lambda,\textrm{Gies}(t))}{p}= & p\alpha_{\mathscr{F}}-\alpha \textrm{dim}(H) \Biggr.
+\delta\sum_{i=1}^{l}d_{i}( r\nu(I_{0},\textrm{dim}(H))-a\textrm{rk}(\mathscr{F}^{i}))+ \\
&+\sum_{i \in I(\underline{e})} \alpha\kappa_{i}\dfrac{r}{e_i}\textrm{dim}(q_{i}(\mathscr{F}(y_{1}^{i})\oplus\mathscr{F}(y_{2}^{i}))-b'_{2}\alpha_{\mathscr{F}}.
\end{align*}
Since $h^{0}(Y,\mathscr{G})\geq p-\textrm{dim}(H)$, we get
\begin{align*}
0\leq &-p\alpha_{\mathscr{G}}+\alpha h^{0}(Y,\mathscr{G}) \Biggr.
+\delta\sum_{i=1}^{l}d_{i}( r\nu(I_{0},\textrm{dim}(H))-a\textrm{rk}(\mathscr{F}^{i}))+ \\
&+\sum_{i\in I(\underline{e})}\alpha\kappa_{i}\dfrac{r}{e_i}\textrm{dim}(q_{i}(\mathscr{F}(y_{1}^{i})\oplus\mathscr{F}(y_{2}^{i}))-b'_{2}\alpha_{\mathscr{F}}.
\end{align*}
and therefore
$
h^{0}(Y,\mathscr{G})\geq \dfrac{p}{\alpha}+ M,
$
$M$ being a constant not depending on $\mathscr{G}$. Note that $p=\alpha n+d+r\chi(\mathscr{O}_{Y})$ and that we can assume $h^{0}(Y,\omega_{Y})\geq h^{0}(Y,\mathscr{G})$. Then, if $n$ is large enough we get a contradiction, so $h^{1}(Y,\mathscr{E}_{(0)}(n))=0$.

Let us show now that $\mathscr{E}_{(0)}$ has no torsion. Assume it has torsion, $\mathcal{T}\subset\mathscr{E}_{(0)}(n)$, supported on the divisors $D_{i}$, and let $T=H^{0}(Y,\mathcal{T})$. Let now $H:=H^{0}(q_{(0)}(n))^{-1}(T)\subset U$. Again, since $t_{(0)}$ is semistable, we have
\begin{align*}
0\leq\mu_{\mathbb{G}}(\lambda,t(0))&= \sum_{i=1}^{l}b_{i}\mu_{\mathbb{G}_{1}^{i}} (\lambda,t_{1,i}(0))+ c\mu_{\mathbb{G}_{2}}(\lambda,t_{2}(0))+ \\
&+\sum_{i\in I(\underline{e})}k_{i}\mu_{\mathcal{G}r}(\lambda,t_{3,i}(0))=\\
&=\sum_{i=1}^{l}b_{i}(-r \textrm{dim}(H))+ca(-\textrm{dim}(H))+\\
&+\sum_{i\in I(\underline{e})}k_{i}(p \textrm{dim}(t_{i0}(H\oplus H)-e_{i} \textrm{dim}(H))=\\
&=\sum_{i\in I(\underline{e})}\kappa_{i}\dfrac{r}{e_i}\alpha(p \textrm{dim}(t_{i0}(H\oplus H)) -\alpha p \textrm{dim}(H)\leq\\
&\leq \sum_{i\in I(\underline{e})}\kappa_{i}\dfrac{r}{e_i}\alpha p \textrm{dim}(T_{D_{i}}) -\sum_{i=1}^{\nu}\alpha p \textrm{dim}(T_{D_{i}})  \\
\end{align*}
Since $\kappa_{i}<\dfrac{e_i}{r}$ we must have $\textrm{dim}(T_{D_{i}})=0$, that is $\mathcal{T}=0$, so $\mathscr{E}_{(0)}$ has no torsion supported on the divisors $D_{i}$. Furthermore, from the last calculation it is clear that there can not be any torsion subsheaf supported outside the divisors $D_{i}$, therefore $\mathscr{E}_{(0)}$ is locally free. 
Thus, the extended family defines a point in $\mathfrak{I}_{\underline{d}}$.
Since the corresponding point in $\mathbb{G}$ lies in the semistable locus we deduce that the extended family lies in the semistable locus, $\mathbb{G}^{s(s)}$, as well and  by Theorem \ref{melon3} we are done.
\end{proof}

%%%%%%%%%%%%%%%%%%
%%%%%%%%%%%%%%%%%%
%%%%%%%%%%%%%%%%%%%%%
%\subsubsection{Construction of the moduli space}
Let  $\underline{d}\in I_{r,d,\delta}$ be as in Section \ref{parameterparten}, Equation (\ref{indicesdegree}), and let $\mathfrak{I}_{\underline{d}}$ be the parameter space constructed in Section \ref{parameterparten}. Over $Y\times\mathfrak{I}_{\underline{d}}$ there is a universal family satisfying the local universal property (follows as in \cite[Proposition 2.8]{Alexander1}). Note also that the natural $\textrm{SL}(U)$ action on $\mathcal{Q}^{0}$, $\mathfrak{h}$ and $\mathscr{G}r_{i}$ determines an action on the space $\mathfrak{I}_{\underline{d}}$, $\Gamma:\textrm{SL}(U)\times\mathfrak{I}_{\underline{d}}\rightarrow\mathfrak{I}_{\underline{d}}$, and that the universal family satisfies the glueing property as well (again it follows as in \cite[Proposition 2.10]{Alexander1}).
Finally, we have

\begin{theorem}\label{theorem E}
There exist a projective scheme $\emph{SGPS}^{(\underline{\kappa},\delta)\text{-ss}}_{r,d,\mathfrak{tp}}$ and an open subscheme  $\emph{SGPS}^{(\underline{\kappa},\delta)\text{-s}}_{r,d,\mathfrak{tp}}$ together with natural transformation
%\begin{equation*}
$
\alpha^{(s)s}: \emph{\textbf{SGPS}}^{(\underline{\kappa},\delta)-(s)s}_{r,d,\mathfrak{tp}} \rightarrow h_{\emph{SGPS}^{(\underline{\kappa},\delta)\text{-(s)s}}_{r,d,\mathfrak{tp}}}
$
%\end{equation*}
with the following propoerties:

1) For every scheme $S$ and every natural transformation 
$\emph{\textbf{SGPS}}^{(\underline{\kappa},\delta)-(s)s}_{r,d,\mathfrak{tp}}\rightarrow h_{\N},$
there exists a unique morphism $\varphi:\emph{SGPS}^{(\underline{\kappa},\delta)\text{-(s)s}}_{r,d,\mathfrak{tp}}\rightarrow S$ with $\alpha'=h(\varphi)\circ\alpha^{(s)s}$.

2) The scheme $\emph{SGPS}^{(\underline{\kappa},\delta)\text{-s}}_{r,d,\mathfrak{tp}}$ is a coarse moduli space for $\emph{\textbf{SGPS}}^{(\underline{\kappa},\delta)-s}_{r,d,\mathfrak{tp}}$.
\end{theorem}
\begin{proof}
%Follows as in the connected case (see \cite{Alexander1}).
We may assume without lost of generality that $e_i\neq 0$ for each $i=1,\hdots,\nu$. 
Consider the Gieseker map $\textrm{Gies}:\mathfrak{I}_{\underline{d}}\hookrightarrow\mathbb{G}$, which is injective and $\textrm{SL}(U)$-equivariant (see Proposition \ref{hastalaminga100}). Consider on $\mathbb{G}$ the polarization
given in Section \ref{ovni}, 
and let 
$\Lcc:=\textrm{Gies}^{*}\mathscr{O}(b_{1},\hdots,b_{l},c,k_{i_1},\hdots,k_{i_{\nu'}}).$ 
From (\cite[Chap.2, \S1]{mum-git}), we know that $\textrm{Gies}^{-1}(\mathbb{G}^{(s)s})=\mathfrak{I}_{\underline{d}}^{\textrm{(s)s}}$, and therefore Theorem \ref{melon3} implies that $\mathfrak{I}_{\underline{d}}^{\textrm{(s)s}}=\mathfrak{I}_{\underline{d}}^{(\underline{\kappa},\delta)\textrm{-(s)s}}$. By Proposition \ref{propergieseker}, we deduce that the restriction of the Gieseker map to the semistable locus is a SL(U)-equivariant injective and proper morphism. Thus

1) the good quotient $\textrm{SGPS}^{(\underline{\kappa},\delta)\text{-ss}}_{r,\underline{d},\mathfrak{tp}}:=\mathfrak{I}_{\underline{d}}^{(\underline{\kappa},\delta)\textrm{-ss}}\slas \textrm{SL}(U)$ exists and is projective,

2) the geometric quotient $\textrm{SGPS}^{(\underline{\kappa},\delta)\text{-s}}_{r,\underline{d},\mathfrak{tp}}:=\mathfrak{I}_{\underline{d}}^{(\underline{\kappa},\delta)\textrm{-s}}/\textrm{SL}(U)$ exists and is an open subscheme of $\textrm{SGPS}^{(\underline{\kappa},\delta)\text{-ss}}_{r,\underline{d},\mathfrak{tp}}$.\\
Define 
$\textrm{SGPS}^{(\underline{\kappa},\delta)\text{-(s)s}}_{r,d,\mathfrak{tp}}:=\coprod_{\underline{d}\in I(r,d,\delta)} \textrm{SGPS}^{(\underline{\kappa},\delta)\text{-(s)s}}_{r,\underline{d},\mathfrak{tp}}.$
Now, 1) and 2) follow from this construction, the local universal property and the glueing property.
\end{proof}

%%%%%%%%%%%%%%%%%%%%%%%%%%%%%%%%%%%
%%%%%%%%%%%%%%%%%%%%%%%%%%%%%%%%%%%

\section{Moduli space for generalized parabolic singular principal bundles}\label{sectionspb}
%\fancyhead[RO,LE]{}

\subsection{The parameter space}\label{singularbundlespace}
Let $r\in\mathbb{N}$, $d\in\mathbb{Z}$, $\underline{e}:=(e_1,\hdots,e_{\nu})\in\mathbb{N}^{\nu}$ with $e_i\leq r$, and $\delta\in\mathbb{Q}_{>0}$. 
In order to prove the existence of a coarse projective moduli space for the moduli functor given in Equation (\ref{modulifunctornormalization})
we need to rigidify the moduli problem. Let $n\in\mathbb{N}$ and $U:=\mathbb{C}^{P(n)}$. Consider the functor 
\begin{equation}\label{rigfunctorbund}
^{rig}\textbf{SPBGPS}(\rho)_{r,d,\underline{e}}^{n}(S)=\left\{   \begin{array}{l} \textrm{isomorphism classes of tuples }(\mathscr{E}_{S},\underline{q}_{S},\tau_{S},g_{S}) \\
 \textrm{where }(\mathscr{E}_{S},\tau_{S}) \textrm{ is a family of singular principal} \\ 
 \textrm{$G$-bundles parametrized by }S \textrm{ with rank $r$} \\
\textrm{and degree $d$, } (\mathscr{E}_{S},\underline{q}_{S})\textrm{ is a family of generalized}  \\
\textrm{parabolic locally free sheaves of type $\underline{e}$ and}\\
g_{S}:U\otimes\mathscr{O}_{S}\rightarrow\pi_{S*}\mathscr{E}_{S}(n)\textrm{ is a morphism such} \\
\textrm{that the induced morphism } \\
U\otimes\mathscr{O}_{Y\times S}(-n)\rightarrow\mathscr{E}_{S}\textrm{ is surjective }
 \end{array}\right\}.
\end{equation}
and let us show that there is a representative for it.

We may assume without loss of generality that $e_i\neq 0$ for all $i=1,\hdots,\nu$. Recall from Proposition \ref{boundparabolicSwamp2} that the family of locally free sheaves $\mathscr{E}$ of rank $r$ and degree $d$ that appear in $(\underline{\kappa},\delta)$-(semi)stable swamps with generalized parabolic structure is bounded. 
In consequence,  there is a natural number $n_{0}\in\mathbb{N}$ such that for $n\geq n_{0}$, $\mathscr{E}(n)$ is globally generated and $H^{1}(Y,\mathscr{E}(n))=0$. Fix $n>\textrm{max}\{n_{0},n_{4}\}$ and $\underline{d}=(d_{1},\hdots,d_{l})\in\mathbb{N}^{l}$ with $d=\sum_{i=1}^{l}d_{i}$, and let $p=r\chi(\mathscr{O}_{Y})+ d+\alpha n$. Let $U$ be the vector space $\mathbb{C}^{\oplus p}$. Denote by $\mathcal{Q}^{0}$ the quasi-projective scheme parametrizing equivalence classes of quotients $\q:U\otimes\pi_{Y}^{*}\mathscr{O}_{Y}(-n)\rightarrow\mathscr{E}$ where $\mathscr{E}$ is a locally free sheaf of uniform multirank $r$ and multidegree $(d_{1},\hdots,d_{l})$ on $Y$, and such that the induced map $U\rightarrow H^{0}(Y,\mathscr{E}(n))$ is an isomorphism. On $\Q\times Y$, we have the morphism,
%\begin{equation*}
$h:S^{\bullet}(V\otimes U\otimes\pi^{*}_{Y}\mathscr{O}_{Y}(-n))\rightarrow S^{\bullet}(V\otimes\mathscr{E}_{\mathcal{Q}^{0}}) \rightarrow S^{\bullet}(V\otimes\mathscr{E}_{\mathcal{Q}^{0}})^{G}$.
%\end{equation*}
Let $s\in\mathbb{N}$ be as in \cite[Theorem 4.2, Remark 4.3]{AMC}. Then
$h(\bigoplus_{i=1}^{s}S^{i}(V\otimes U\otimes\pi_{Y}\mathscr{O}_{Y}(-n)))$,
contains a set of generators of $S^{\bullet}(V\otimes\mathscr{E}_{\mathcal{Q}^{0}})^{G}$. Observe that every morphism $k\colon\oplus_{i=1}^{s}S^{i}(V\otimes U\otimes\mathscr{O}_{Y}(-n))\rightarrow\mathscr{O}_{Y}$ breaks into a family of morphisms
%\begin{equation*}
$k_{i}:S^{i}(V\otimes U)\otimes\mathscr{O}_{Y}(-in)\simeq S^{i}(V\otimes U\otimes\mathscr{O}_{Y}(-n))\rightarrow\mathscr{O}_{Y}$
%\end{equation*}
and therefore into morphisms
%\begin{equation*}
$k_{i}:S^{i}(V\otimes U)\overset{\Delta}{\hookrightarrow }S^{i}(V\otimes U)\otimes \mathbb{C}^{\oplus l}\rightarrow H^{0}(Y,\mathscr{O}_{Y}(in))$,
%\end{equation*}
$\Delta$ being the diagonal morphism.  From this point onwards we can proceed as in \cite[\S 6.1]{AMC} and we end up with a closed subscheme $\mathbb{D}\subset \mathcal{Q}^{*}$ together with a universal family $(\mathscr{E}_{\mathbb{D}},\tau_{\mathbb{D}})$ of singular principal $G$-bundles of uniform multirank $r$ and multidegree $(d_{1},\hdots,d_{l})$.
In order to include the parabolic structure as well we need to consider the Grassmannians $\mathscr{G} r_{i}:=\textrm{Grass}_{e_{i}}(U^{\oplus 2})$ of $e_{i}$ dimensional quotients of $U^{\oplus 2}$. Define 
$Z:=\mathbb{D}\times \mathscr{G} r_{1}\times\hdots\times\mathscr{G} r_{\nu},$
and denote by $c_{i}:Z\rightarrow\mathscr{G} r_{i}$ the projection onto the $i$th Grassmannian. Consider the pullback of the universal quotient of the $i$th Grassmannian to $Z$,
%\begin{equation*}
$
q_{Z}^{i}: U^{\oplus 2}\otimes\mathscr{O}_{Z}\rightarrow R_{Z},
$
%\end{equation*}
and take the direct sum
$q_{Z}: U^{\oplus 2\nu}\otimes\mathscr{O}_{Z}\rightarrow \bigoplus_{1}^{\nu} R_{Z}$.
Denote by $q_{Z}$, $\mathscr{E}_{Z}$ and $\tau_{Z}$ the pullbacks to $Z\times Y$ of the corresponding objects over $\mathbb{D}$. Consider the morphism $\pi^{i}:Z\times\{y^{i}_{1}, y^{i}_{2}\}\rightarrow Z\times\{x_{i}\}\simeq Z$.
and look at the following commutative diagram
For each $i$, there are quotients
$f_{i}:U^{\oplus 2}\times\mathscr{O}_{Z}\rightarrow \pi^{i}_{*}(\mathscr{E}_{Z}|_{y^{i}_{1},y^{i}_{2}})$ and we can form
$f:=\oplus(f_{i}):U^{\oplus 2\nu}\times\mathscr{O}_{Z}\rightarrow \bigoplus\pi^{i}_{*}(\mathscr{E}_{Z}|_{y^{i}_{1},y^{i}_{2}})$.
Consider the following diagram,
$$
\xymatrix{
0\ar[r]& \textrm{Ker}(f)\ar[r]\ar@{.>}[rd]^{q'} & 
U^{\oplus 2\nu}\times\mathscr{O}_{Z} \ar[r]^{f}\ar[d]^{q_{Z}}& \bigoplus\pi^{i}_{*}(\mathscr{E}_{Z}|_{y^{i}_{1},y^{i}_{2}}) \ar[r]& 0 \\
 & &  \bigoplus R_{Z}  & &.
}
$$
Denote by $\mathfrak{M}_{\underline{d}}(G)\subset Z$ the closed subscheme given by the zero locus of the morphism $q'$ (see \cite[lemma 3.1]{sols}). Then, the restriction of $q_{Z}$ to $\mathfrak{M}_{\underline{d}}(G)$  factorizes 
$$
\xymatrix{
 \bigoplus\pi_{*}^{i}(\mathscr{E}_{Z}|_{y^{i}_{1},y^{i}_{2}})|_{\mathfrak{M}_{\underline{d}}(G)}\ar@{=}[d] &&    \bigoplus R_{\mathfrak{M}_{\underline{d}}(G)}\ar@{=}[d] \\
\bigoplus\pi^{i}_{\mathfrak{M}_{\underline{d}}(G)*}(\mathscr{E}_{\mathfrak{M}_{\underline{d}}(G)}|_{y^{i}_{1},y^{i}_{2}}) \ar[rr]^-{q_{\mathfrak{M}_{\underline{d}}(G)}}&& \bigoplus R_{Z}|_{\mathfrak{M}_{\underline{d}}(G)}.
}
$$
Since $f$ and $q_{Z}$ are diagonal morphisms we deduce that $q_{\mathfrak{M}_{\underline{d}}(G)}$ is also diagonal. Therefore $q_{\mathfrak{M}_{\underline{d}}(G)}$ is determined by $\nu$ morphisms
%\begin{equation*}
$q_{\mathfrak{M}_{\underline{d}}(G)}^{i}:\pi^{i}_{\mathfrak{M}_{\underline{d}}(G)*}(\mathscr{E}_{\mathfrak{M}_{\underline{d}}(G)}|_{y^{i}_{1},y^{i}_{2}}) \rightarrow R_{\mathfrak{M}_{\underline{d}}(G)}.$
%\end{equation*}
Denote by $(\mathscr{E}_{\mathfrak{M}_{\underline{d}}(G)},\tau_{\mathfrak{M}_{\underline{d}}(G)})$ the restriction of $(\mathscr{E}_{Z},\tau_{Z})$ to $\mathfrak{M}_{\underline{d}}(G)$. Then $(\mathscr{E}_{\mathfrak{M}_{\underline{d}}(G)},\underline{q}_{\mathfrak{M}_{\underline{d}}(G)},\tau_{\mathfrak{M}_{\underline{d}}(G)})$ is a universal family of singular principal $G$-bundles with generalized parabolic structure.

\begin{theorem}\label{GILI}
The functor $^{rig}\emph{\textbf{SPBGPS}}(\rho)_{r,d,\underline{e}}^{n}$ is representable.
\end{theorem}
\begin{proof}
Follows from the construction of $\mathfrak{M}_{\underline{d}}(G)$ and taking the disjoint union over all the possible multidegrees as in Theorem \ref{theorem E}, which we denote by $\mathfrak{M}(G)$.
\end{proof}

\subsection{The moduli space}

Recall from Proposition \ref{boundparabolicSwamp2} that the family of locally free sheaves $\mathscr{E}$ of fixed degree and rank which appears in a $(\underline{\kappa},\delta)$-(semi)stable swamp with generalized parabolic structure is bounded.
As a consequence,  there is a natural number $n_0\in\mathbb{N}$ such that for $n\geq n_{0}$, $\mathscr{E}(n)$ is globally generated and $h^{1}(Y,\mathscr{E}(n))=0$. Fix such natural number $n$ and consider the functors $^{rig}\textbf{SGPS}^{n}_{r,d,\mathfrak{tp}}$ and $^{rig}\textbf{SPBGPS}(\rho)_{r,d,\underline{e}}^{n}$ given in Equation (\ref{rigfunctorswamps}) and Equation (\ref{rigfunctorbund}) respectively.
Note that there is a natural $\textrm{GL}(U)$ action on the space $\mathfrak{M}(G)$,
$\Gamma:\textrm{GL}(U)\times\mathfrak{M}(G)\rightarrow\mathfrak{M}(G)$.
We can view  this $\textrm{GL}(U)$-action as a $(\mathbb{C}^{*}\times \textrm{SL}(U))$-action. Thus, we will construct the quotient of $\mathfrak{M}(G)$ by $\textrm{GL}(U)$ in two steps, considering the actions of $\mathbb{C}^{*}$ and $\textrm{SL}(U)$ separately. 
Consider the action of $\mathbb{C}^{*}$ on $^{rig}\textbf{SPBGPS}(\rho)_{r,d,\underline{e}}^{n}$. Let $\mathfrak{tp}=(a,b,0,\mathscr{O}_{Y},\underline{e})$, where $a$ and $b$ are as in \cite[Theorem 5.5]{AMC}. The map given in Equation (\ref{swamp-bundle}) induces an injective $\mathbb{C}^{*}$-invariant natural transformation
\begin{equation*}
^{rig}\textbf{SPBGPS}(\rho)_{r,d,\underline{e}}^{n}\hookrightarrow \  ^{rig}\textbf{SGPS}^{n}_{r,d,\mathfrak{tp}},
\end{equation*}
which in turn induces a $\textrm{SL}(U)$-equivariant injective and proper morphism, 
$$\beta:\mathfrak{M}(G)\slas \mathbb{C}^{*}\hookrightarrow \mathfrak{I}_{r,d,\mathfrak{tp}}=\coprod_{\underline{d}\in I} \mathfrak{I}_{\underline{d}}.$$
Furthermore, the universal family on $\mathfrak{M}(G)$ satisfies the local universal property as well as the glueing property. We finally have
\begin{theorem}\label{CC}
There is a projective scheme $\emph{SPBGPS}(\rho)_{r,d,\underline{e}}^{(\underline{\kappa},\delta)-\emph{ss
}}$ and an open subscheme $\emph{SPBGPS}(\rho)_{r,d,\underline{e}}^{(\underline{\kappa},\delta)-\emph{s}}\subset \emph{SPBGPS}(\rho)_{r,d,\underline{e}}^{(\underline{\kappa},\delta)-\emph{ss}}$ together with a natural tranformation
%\begin{equation*}
$
\alpha^{\emph{(s)s}}:\emph{\textbf{SPBGPS}}(\rho)_{r,d,\underline{e}}^{(\underline{\kappa},\delta)\emph{-(s)s}}\rightarrow h_{\emph{SPBGPS}(\rho)_{r,d,\underline{e}}^{(\underline{\kappa},\delta)\emph{-(s)s}}}
$
%\end{equation*}
with the following properties:

1) For every scheme $S$ and every natural transformation 
$\alpha': \emph{\textbf{SPBGPS}}(\rho)_{r,d,\underline{e}}^{(\underline{\kappa},\delta)\emph{-(s)s}} \rightarrow h_{S},$ there exists a unique morphism $\varphi:\emph{SPBGPS}_{r,d,\underline{e}}^{(\underline{\kappa},\delta)\emph{-(s)s}}(\rho)\rightarrow S$ with $\alpha'=h(\varphi)\circ\alpha^{\emph{(s)s}}$.

2) The scheme $\emph{SPBGPS}_{r,d,\underline{e}}^{(\underline{\kappa},\delta)\emph{-s}}(\rho)$ is a coarse quasi-projective moduli space for
the moduli functor $\emph{\textbf{SPBGPS}}_{r,d,\underline{e}}^{(\underline{\kappa},\delta)\emph{-s}}(\rho)$.
\end{theorem}
\begin{proof}
Considering the linearized invertible sheaf $\Lcc$ given in the proof of Theorem \ref{theorem E} and defining $\Lcc':=\beta^{*}\Lcc$,  it follows as in the connected case (see \cite{Alexander1}).
\end{proof}

\section{Application to principal bundles on reducible nodal curves}\label{applicationsection}
%\fancyhead[RO,LE]{}

Let $X$ be a projective nodal curve with nodes $x_{1},\hdots,x_{\nu}$ and $l$ irreducible components, and $\pi\colon Y=\coprod_{i=1}^{l} Y_{i}\rightarrow X$ its normalization. Let $\mathscr{O}_{X}(1)$ be an ample invertible sheaf on $X$ and denote by $\mathscr{O}_{Y}(1)$ the ample invertible sheaf obtained by pulling $\mathscr{O}_{X}(1)$ back to $Y$. As usual, $h$ is the degree of $\mathscr{O}_{Y}(1)$, $y_{1}^{i},y_{2}^{i}$ are the points in the preimage of the $i$th nodal point $x_{i}$, $D_{i}=y_{1}^{i}+y_{2}^{i}$ are the corresponding divisor on $Y$ and $D=\sum D_{i}$ is the total divisor.

\subsection{Torsion free sheaves over a reducible nodal curve}

Let $\mathscr{F}$ be a torsion free sheaf on $X$ of rank $r$, that is, of uniform multirank $r$. C. S. Seshadri showed (see \cite[Chapter 8]{Seshadri}) that for each nodal point $x$ (regardless of how many components this point lies on), there is a natural number $0\leq l\leq r$ such that 
$
\mathscr{F}_{x}\simeq\mathscr{O}_{X,x}^{l}\oplus\mathfrak{m}_{x}^{r-l}.
$
Then, it is said that
a torsion free sheaf of rank $r$ is of type $\underline{l}=(l_{1},\hdots,l_{\nu})$ if $\mathscr{F}_{x_{i}}\simeq \mathscr{O}_{X,x_{i}}^{l_{i}}\oplus\mathfrak{m}_{x_{i}}^{r-l_{i}}$ at the $i$th nodal point.

If $\mathscr{F}$ be a torsion free sheaf on $X$ of rank $r$ and of type $\underline{l}$, then the canonical map $\alpha:\mathscr{F}\rightarrow\pi_{*}\pi^{*}(\mathscr{F})$ is injective, and $\mathscr{T}:=\textrm{Coker}(\alpha)$ is a torsion sheaf supported on the nodes. A short calculation shows that 
$\textrm{length}(\mathscr{T})=\sum_{i=1}^{\nu}(2r-l_{i})$.
and
\begin{equation}\label{degreespullback}
\begin{split}
\textrm{deg}(\pi^{*}\mathscr{F})&=\textrm{deg}(\mathscr{F})+r\nu-\sum l_{i}, \\
\textrm{deg}(T(\mathscr{F}))&=2(r\nu-\sum l_{i}),
\end{split}
\end{equation}
$T(\mathscr{F})$ being the torsion subsheaf of $\pi^{*}(\mathscr{F})$ (see \cite{Lange} for the irreducible case).

\begin{proposition}\label{canonical-bundle}
If $\mathscr{F}$ is a torsion free sheaf of rank $r$ and type $\underline{l}=(l_{1},\hdots,l_{\nu})$ on $X$, then the natural  morphism
%\begin{equation*}
$\beta:\mathscr{F}\hookrightarrow\pi_{*}(\mathscr{E}_{0})$,
%\end{equation*}
where $\mathscr{E}_{0}:=\pi^{*}(\mathscr{F})/T(\mathscr{F})$, is injective and  $\emph{length}(\emph{Coker}(\beta))=l:=\sum l_{i}$. Furthermore, $\emph{Coker}(\beta)=\bigoplus_{i=1}^{\nu}\mathbb{C}_{x_{i}}^{l_{i}}$.
\end{proposition}
\begin{proof}
Let $\mathscr{F}$ be a torsion free sheaf on the nodal curve $X$ and let $T(\mathscr{F})$ be the torsion subsheaf of $\pi^{*}(\mathscr{F})$. Consider the natural morphism $\beta:\mathscr{F}\rightarrow \pi_{*}(\pi^{*}(\mathscr{F})/T(\mathscr{F}))$. This is injective at every smooth point so it is injective since $\mathscr{F}$ is torsion free.
Consider  now the  exact sequence
\begin{equation}\label{coker}
0\rightarrow\mathscr{F}\hookrightarrow\pi_{*}(\pi^{*}(\mathscr{F})/T(\mathscr{F}))\rightarrow \textrm{Coker}(\beta)\rightarrow 0.
\end{equation}
Then,  we have
$\chi(\pi^{*}(\mathscr{F})/T(\mathscr{F}))=\chi(\mathscr{F})+\textrm{length}(\textrm{Coker}(\beta))$
and, therefore,
$r\chi(\mathscr{O}_{Y})+\textrm{deg}(\pi^{*}(\mathscr{F})/T(\mathscr{F}))=r\chi(\mathscr{O}_{X})+\textrm{deg}(\mathscr{F})+\textrm{length}(\textrm{Coker}(\beta))$.
However $\chi(\mathscr{O}_{Y})-\chi(\mathscr{O}_{X})=\nu$, so
$\textrm{length}(\textrm{Coker}(\beta))=r\nu+\textrm{deg}(\pi^{*}(\mathscr{F})/T(\mathscr{F}))-\textrm{deg}(\mathscr{F})$
and applying Equation (\ref{degreespullback}) we get the result.
\end{proof}

\begin{corollary}\label{length}
Let $\mathscr{F}$ be a torsion free sheaf of rank $r$ and type $\underline{l}=(l_{1},\hdots,l_{\nu})$ on $X$. Suppose there exists a locally free sheaf $\mathscr{E}$ on $Y$ of the same rank and an injection $i:\mathscr{F}\hookrightarrow\pi_{*}\mathscr{E}$. Then $\emph{length}(\emph{Coker}(i))=e$ if and only if $\emph{length}(\emph{Coker}(\pi_{*}(\lambda)))=e-l$, where $l=\sum l_{i}$.
\end{corollary}
\begin{proof}
Let $\mathscr{F}$  be a torsion free sheaf of rank $r$ on $X$ and suppose there exists a locally free sheaf of rank $r$, $\mathscr{E}$, on the normalization and an injection $i:\mathscr{F}\hookrightarrow\pi_{*}(\mathscr{E})$. Then, there is an injection $\lambda:\mathscr{E}_{0}\hookrightarrow \mathscr{E}$ such that $\pi_{*}(\lambda)\circ\beta=i$.
From the above observation, it follows that $\textrm{Coker}(i)/\textrm{Coker}(\beta)\simeq \textrm{Coker}(\pi_{*}(\lambda))$.
Hence, we deduce that $\textrm{length}(\textrm{Coker}(\pi_{*}(\lambda)))=\textrm{length}(\textrm{Coker}(i))-\textrm{length}(\textrm{Coker}(\beta))$. Since $\textrm{length}(\textrm{Coker}(i))=e$ and $\textrm{length}(\textrm{Coker}(\beta))=l$, we can conclude using Proposition \ref{canonical-bundle}.
\end{proof}

\subsection{Descending singular principal $G$-bundles}

Let $r\in\mathbb{N}$, $d\in\mathbb{Z}$ and $\underline{e}:=(e_1,\hdots,e_{\nu})\in\mathbb{N}^{\nu}$ with $e_i\leq r$.
Let $(\mathscr{E},\underline{q},\tau)$ be a singular principal $G$-bundle with generalized parabolic structure on $Y$ with rank $r$, degree $d$ and type $\underline{e}$. Consider the natural surjection 
%\begin{equation*}
$
\textrm{ev}_{D}=\oplus \textrm{ev}_{i}:\mathscr{E}\rightarrow\mathscr{E}|_{D}=\bigoplus \mathscr{E}|_{D_{i}}
$
%\end{equation*}
and take the push-forward,
%\begin{equation*}
$\pi_{*}(\textrm{ev}_{D}):\pi_{*}(\mathscr{E})\rightarrow\pi_{*}(\mathscr{E}|_{D})$.
%\end{equation*}
Since $\pi_{*}(\mathscr{E}|_{D})$ is precisely the vector space $\bigoplus(\mathscr{E}(y_{1}^{i})\oplus\mathscr{E}(y_{2}^{i}))$ supported on the nodes, we can consider $R=\bigoplus R_{i}$ as a skycraper sheaf supported on the nodes and compose $\pi_{*}(\textrm{ev}_{D})$ with $q$ to get the morphism
%\begin{equation*}
$q\circ\pi_{*}(\textrm{ev}_{D}):\pi_{*}(\mathscr{E})\rightarrow R\rightarrow 0$.
%\end{equation*}
Defining $\mathscr{F}=\textrm{Ker}(q\circ\pi_{*}(\textrm{ev}_{D}))$, we get an exact sequence
\begin{equation}\label{induced}
0\rightarrow\mathscr{F}\hookrightarrow\pi_{*}(\mathscr{E})\overset{p}{\rightarrow} R\rightarrow 0
\end{equation}
where $\mathscr{F}$ is a torsion free sheaf of rank $r$ and degree $d+\sum_{i=1}^{\nu}(r-e_{i})$,  and $R$ has length $\textrm{length}(R):=e_{1}+\hdots+e_{\nu}$.

It remains to construct $\tau':\underline{\textrm{Spec}}(\mathscr{F}\otimes V)^{G}\rightarrow\mathscr{O}_{X}$ from the data $(\mathscr{E},\underline{q},\tau)$. Consider the 
canonical isomorphism,
%\begin{equation}\label{pullback}
$\pi^{*}(\underline{\textrm{Spec}}(\mathscr{F}\otimes V)^{G})\simeq\underline{\textrm{Spec}}(\pi^{*}(\mathscr{F})\otimes V)^{G}$.
%\end{equation}
Now, the identity map $\pi_{*}\mathscr{E}\rightarrow \pi_{*}\mathscr{E}$ induces a morphism $\pi^{*}\pi_{*}\mathscr{E}\rightarrow \mathscr{E}$ by adjunction and therefore a morphism of algebras $\pi^{*}S^{\bullet}(V\otimes \pi_{*}\mathscr{E})^{G}\rightarrow S^{\bullet}(V\otimes \mathscr{E})^{G}$ which, in turn, induces a morphism of algebras $S^{\bullet}(V\otimes \pi_{*}\mathscr{E})^{G}\rightarrow \pi_{*}S^{\bullet}(V\otimes \mathscr{E})^{G}$ again by adjunction. This induces a diagram
\begin{equation*}\label{adjuncion}
\xymatrix{
& S^{\bullet}(V\otimes \mathscr{F})^{G}\ar[r]\ar[rd]^{\tau'} & S^{\bullet}(V\otimes\pi_{*}\mathscr{E})^{G} \ar[d]^{\hat{\tau}} & &\\
0\ar[r]&\mathscr{O}_{X} \ar@{^(->}[r] & \pi_{*}\mathscr{O}_{Y}\ar[r] & \bigoplus_{i=1}^{\nu}\mathbb{C}_{x_{i}}\ar[r]& 0 
}
\end{equation*}

\begin{remark}\label{funtorial-des}
Let $(\mathscr{E},\underline{q})$ be a generalized parabolic locally free sheaf of rank $r$, degree $d$ and type $\underline{e}'=(e'_1,\hdots,e'_{\nu})$. For each $i=1,\hdots,\nu$, denote by $K_{i}$ the kernel of the $i$th parabolic structure $\mathscr{E}(y_{1}^{i})\oplus\mathscr{E}(y_{2}^{i})\rightarrow R_{i}$ and by $C_{1}^{i}$ (resp. $C_{2}^{i}$) the kernel of the induced linear map $K_{i}\rightarrow \mathscr{E}(y_{1}^{i})$ (resp. $K_{i}\rightarrow \mathscr{E}(y_{2}^{i})$). From \cite[Proposition 3.7]{bosle-multiple}, it follows that the associated torsion free sheaf $\mathscr{F}$ satisfies $\mathscr{F}_{x_{i}}\simeq \mathscr{O}_{X}^{e_{i}}\oplus \mathfrak{m}_{x_{i}}^{r-e_{i}}$, where $e_{i}=2r-e'_{i}-\textrm{dim}(C_{1}^{i})-\textrm{dim}(C_{2}^{i})$.
\end{remark}

\begin{definition}\label{descienden}
Let $r\in\mathbb{N}$, $d\in\mathbb{Z}$ and $\underline{e}:=(e_1,\hdots,e_{\nu})\in\mathbb{N}^{\nu}$ with $e_i\leq r$.
A descending $G$-bundle of rank $r$, degree $d$ and type $\underline{e}$ on $Y$ is a singular principal $G$-bundle with generalized parabolic structure of rank $r$, degree $d$ and type $\underline{e}$, $(\mathscr{E},\underline{q},\tau)$, such that $\tau'$ takes values in $\mathscr{O}_{X}\subset\pi_{*}(\mathscr{O}_{Y})$. 
\end{definition}

\begin{definition}
Let $r\in\mathbb{N}$, $d\in\mathbb{Z}$, $\underline{e}:=(e_1,\hdots,e_{\nu})\in\mathbb{N}^{\nu}$, and let $\delta\in\mathbb{Q}_{>0}$. For each $i\in I(\underline{e})$ fix $\kappa_{i}\in(0,\dfrac{e_{i}}{r})\cap\mathbb{Q}$.
A descending $G$-bundle is $(\underline{\kappa},\delta)$-(semi)stable if it is as singular principal $G$-bundle with generalized parabolic structure.
\end{definition}
A family of descending $G$-bundles parametrized by a scheme $S$ is defined in the obvious way, and we can consider the moduli functor, 
$$
\textbf{D}(\rho)_{r,d,\underline{e}}^{(\underline{\kappa},\delta)\textrm{-(s)s}}(S)=\left\{   \begin{array}{l}\textrm{isomorphism classes of families of}\\ 
(\underline{\kappa},\delta)\textrm{-(semi)stable} \textrm{ descending} \\ 
G\textrm{-bundles on $Y$ parametrized by} \\ 
\textrm{$S$ with rank $r$}\textrm{ degree }d \textrm{ and type }\underline{e}\\
\end{array}  \right\}.
$$

Then one can show the next theorem following a similar argument as given for proving Theorem \ref{CC} and \cite[Main Theorem]{Alexander1}.
\begin{theorem}\label{D}
There exist a projective scheme $\emph{D}(\rho)_{r,d,\underline{e}}^{(\underline{\kappa},\delta)\text{-ss}}$ and an open subscheme $\emph{D}(\rho)_{r,d,\underline{e}}^{(\underline{\kappa},\delta)\text{-s}}\subset \emph{D}(\rho)_{r,d,\underline{e}}^{(\underline{\kappa},\delta)\text{-ss}}$ together with a natural tranformation
%\begin{equation*}
$
\alpha^{(s)s}:\emph{\textbf{D}}(\rho)_{r,d,\underline{e}}^{(\underline{\kappa},\delta)\textrm{-(s)s}}\rightarrow h_{\emph{D}(\rho)^{(\underline{\kappa},\delta)\text{-(s)s}}}
$
%\end{equation*}
with the following properties:

1) For any scheme $S$ and any natural transformation $\alpha': \emph{\textbf{D}}(\rho)_{r,d,\underline{e}}^{(\underline{\kappa},\delta)\emph{-(s)s}} \rightarrow h_{S}$, there exists a unique morphism $\varphi:\emph{D}(\rho)_{r,d,\underline{e}}^{(\underline{\kappa},\delta)\textrm{-(s)s}}\rightarrow S$ with $\alpha'=h(\varphi)\circ\alpha^{(s)s}$.

2) The scheme $\emph{D}(\rho)_{r,d,\underline{e}}^{(\underline{\kappa},\delta)\textrm{-s}}$ is a coarse moduli space for the moduli functor $\emph{\textbf{D}}(\rho)_{r,d,\underline{e}}^{(\underline{\kappa},\delta)\textrm{-s}}$.
\end{theorem}

\subsection{Relation to the moduli space of principal \emph{G}-bundles over a reducible nodal curve. Specializations}\label{inducedfiltration}

Let $r\in\mathbb{N}$, $d\in\mathbb{Z}$ and $\underline{e}:=(e_1,\hdots,e_{\nu})\in\mathbb{N}^{\nu}$ with $e_i\leq r$.
Let  $(\mathscr{E},\underline{q},\tau)$ be a descending $G$-bundle of rank $r$, degree $d$ and type $\underline{e}$, and $(\mathscr{F},\tau')$ the induced singular principal $G$-bundle. Recall that both sheaves, $\mathscr{E}$ and $\mathscr{F}$, are related through the exact sequence given in Equation (\ref{induced})
where the morphism $p$ factorizes over the surjection $q:\pi_{*}(\mathscr{E}|_{D})\rightarrow R$.  For any subsheaf $\G\subset\mathscr{E}$, the image of  $p$ restricted to $\pi_{*}(\G)\subset\pi_{*}(\mathscr{E})$ is precisely $\bigoplus_{i=1}^{\nu}q_{i}(\G(y_{1}^{i})\oplus\G(y_{2}^{i}))$. Therefore we can construct the following diagram
\begin{equation}\label{saturatedsub}
\xymatrix{
0\ar[r] & \mathscr{F}\ar@{^(->}[r] & \pi_{*}(\mathscr{E})\ar[r] & R\ar[r] & 0\\
0\ar[r] & \textrm{Ker}(p')\ar@{^(->}[r]\ar@{^(-->}[u] & \pi_{*}(\G)\ar[r]\ar@{^(->}[u] &  \bigoplus_{i=1}^{\nu}q_{i}(\G(y_{1}^{i})\oplus\G(y_{2}^{i}))\ar[r]\ar@{^(->}[u] & 0
}
\end{equation}
and we define $S(\G):=\textrm{Ker}(p')$. If $\G$ is saturated then $S(\G)$ is clearly saturated. This construction allows us to attach to any weighted filtration $(\mathscr{E}_{\bullet},\underline{m})$ of $\mathscr{E}$ by saturated sheaves a weighted filtration $(S(\mathscr{E}_{\bullet}),\underline{m})$ of $\mathscr{F}$ by saturated sheaves. Moreover, any saturated subsheaf can be constructed from a saturated subsheaf of $\mathscr{E}$ (follows as in the connected case \cite{Alexander1}).

In what follows, we will use the notation $\underline{\kappa}(\underline{e})$ for $(\dfrac{e_{i_1}}{r},\hdots,\dfrac{e_{i_{\nu'}}}{r})$, where $i_1,\hdots,i_{\nu'}$ are the indices in $I(\underline{e})$.

\begin{proposition}\label{propimp1}
Let $(\mathscr{E},\underline{q},\tau)$ be a descending $G$-bundle of rank $r$ degree $d$ and type $\underline{e}$ and $(\mathscr{F},\tau')$ the induced singular principal $G$-bundle on $X$. %and $\underline{\kappa}=(\kappa_{1},\hdots,\kappa_{\nu})$. 
Then, $(\mathscr{F},\tau')$ is $\delta$-(semi)stable if and only if $(\mathscr{E},\underline{q},\tau)$ is a $(\underline{\kappa}(\underline{e}),\delta)$-(semi)stable $G$-bundle with a generalized parabolic structure.

\end{proposition}

\begin{proof}
This follows as in the irreducible case \cite[Proposition 5.2.2]{Alexander1}
\end{proof}

\begin{proposition}\label{propimp2}

Let $r\in\mathbb{N}$, $d\in\mathbb{Z}$ and $\underline{e}:=(e_1,\hdots,e_{\nu})\in\mathbb{N}^{\nu}$ with $e_i\leq r$.
There exists $\epsilon\in\mathbb{R}\cap (0,1)$, such that for any $\underline{\kappa}$ with  $\dfrac{e_{i}}{r}-\epsilon<\kappa_{i}<\dfrac{e_{i}}{r}$, any integral parameter $\delta$, and any singular principal $G$-bundle $(\mathscr{E},\underline{q},\tau)$ with a generalized parabolic structure of rank $r$, degree $d$ and type $\underline{e}$, we have

1) if $(\mathscr{E},\underline{q},\tau)$ is $(\underline{\kappa},\delta)$-semistable, then it is $(\underline{\kappa}(\underline{e}),\delta)$-semistable,

2) if $(\mathscr{E},\underline{q},\tau)$ is $(\underline{\kappa}(\underline{e}),\delta)$-stable, then it is $(\underline{\kappa},\delta)$-stable.
\end{proposition}
\begin{proof}
Recall that the $(\underline{\kappa},\delta)$-(semi)stability condition for a singular principal $G$-bundle with a generalized parabolic structure has to be checked just for the weighted filtrations $(\mathscr{E}^{\bullet},\underline{m})$ of $\mathscr{E}$ for which $m_{i}<A$ for  suitable constant $A$ depending only on the numerical input data (see Remark \ref{finiteset}). This implies that we can find a natural number $n$ such that
%\begin{equation*}
$P_{\underline{\kappa}(\underline{e})}(\mathscr{E}^{\bullet},\underline{m})+\delta\mu(\mathscr{E}^{\bullet},\underline{m},\tau)\in\mathbb{Z}[\dfrac{1}{n}]$
%\end{equation*}
for all such weighted filtrations. A short calculation shows that for every generalized parabolic bundle $(\mathscr{E},\underline{q})$ and every weighted filtration $(\mathscr{E}_{\bullet},\underline{m})$ we have
%\begin{equation*}
$P_{\underline{\kappa}(\underline{e})}(\mathscr{E}^{\bullet},\underline{m})-P_{\underline{\kappa}}(\mathscr{E}^{\bullet},\underline{m})\leq \nu r\epsilon A \alpha^{2} $.
In fact we can also show that $P_{\underline{\kappa}(\underline{e})}(\mathscr{E}^{\bullet},\underline{m})-P_{\underline{\kappa}}(\mathscr{E}^{\bullet},\underline{m})\geq - \nu r\epsilon A \alpha^{2} $. 
Take $\epsilon$ so that the inequality $\nu r\epsilon A \alpha^{2} < \dfrac{1}{n}$ holds. Now 1) and 2) follow by a similar argument as given in \cite[Proposition 5.2.3.]{Alexander1}.
\end{proof}

Let $r\in\mathbb{N}$, $d\in\mathbb{Z}$ and $\underline{e}\in J(r):=\{(e_1,\hdots,e_{\nu})\in\mathbb{N}^{\nu}| e_i\leq r\}$. Denote by $\mathfrak{D}_{r,d(\underline{e},r),\underline{e}}$ the set of isomorphism classes of descending $G$-bundles over $Y$ with rank $r$ type $\underline{e}$ and degree $d(\underline{e},r)=d-\sum_{i=1}^{\nu}(r-e_{i})$, and by $\mathfrak{SPB}_{r,d,\underline{e}}$ the set of isomorphism classes of singular principal $G$-bundles over $X$ of rank $r$ degree $d$ and type $\underline{e}$. From Corollary \ref{length}, it follows that there is a map
$
\Theta_{\underline{e}}:\mathfrak{D}_{r,d(\underline{e},r),\underline{e}}\longrightarrow \bigcup_{\underline{e}'\leq \underline{e}}\mathfrak{SPB}_{r,d,\underline{e}'}
$

\begin{theorem}\label{equivsets}
$\Theta_{\underline{e}}$ induces a bijection $\Theta_{\underline{e}}^{-1}(\mathfrak{SPB}_{r,d,\underline{e}})\rightarrow \mathfrak{SPB}_{r,d,\underline{e}}$.
%\item $\Theta_{\underline{e}}$ is surjective
\end{theorem}
\begin{remark}\label{funtorial-des2}
From Remark \ref{funtorial-des} it follows that $\Theta_{\underline{e}}^{-1}(\mathfrak{SPB}_{r,d,\underline{e}})$ consists of descending singular principal $G$-bundles $(\mathscr{E},\underline{q},\tau)\in \mathfrak{D}_{r,d(\underline{e},r),\underline{e}}$ satisfying $\textrm{dim}(C_{1}^{i})+\textrm{dim}(C_{2}^{i})=2(r-e_{i})$ for $i=1,\hdots,\nu$.
\end{remark}
\begin{proof}
\begin{enumerate}
\item Let $(\mathscr{F},\tau)$ be a singular principal $G$-bundle of rank $r$, degree $d$ and type  $\underline{e}$, and consider the exact sequence 
\begin{equation}\label{exact_norm}
\xymatrix{
0\ar[r] & T(\mathscr{F})\ar[r] & \pi^{*}(\mathscr{F})\ar[r] & \mathscr{E}_{0}= \pi^{*}\mathscr{F}/T(\mathscr{F}) \ar[r] & 0 \ .
}
\end{equation}
Since $S^{\bullet}(V\otimes\pi^{*}\mathscr{F})^{G}\rightarrow S^{\bullet}(V\otimes\mathscr{E}_{0})^{G}\rightarrow 0$
is still surjective we find a closed immersion
$\textrm{\underline{Spec}}(S^{\bullet}(V\otimes\mathscr{E}_{0})^{G}) \hookrightarrow \textrm{\underline{Spec}}(S^{\bullet}(V\otimes\pi^{*}\mathscr{F})^{G})$.
We have the following diagram
$$
\xymatrix{
\textrm{\underline{Spec}}(S^{\bullet}(V\otimes\mathscr{E}_{0})^{G})\ar@{^(->}[r]& \textrm{\underline{Spec}}(S^{\bullet}(V\otimes\pi^{*}\mathscr{F})^{G})\ar[r]\ar[d]& \textrm{\underline{Spec}}(S^{\bullet}(V\otimes\mathscr{F})^{G})\ar[d] \\
 & Y\ar[r]^{\pi}\ar@{-->}@/^{5mm}/[u]^{\pi^{*}(\tau)} & X\ar@{-->}@/^{5mm}/[u]^{\tau} \ ,
}
$$
The morphism
$\pi^{*}(\tau):\pi^{*}(S^{\bullet}(V\otimes\mathscr{F})^{G})=S^{\bullet}(V\otimes\pi^{*}\mathscr{F})^{G} \rightarrow \pi^{*}\mathscr{O}_{X}=\mathscr{O}_{Y}$
is the one that we obtain by adjunction when we take the composition of $S^{\bullet}(V\otimes\mathscr{F})^{G}\rightarrow\mathscr{O}_{X}$ with the natural inclusion of rings $\mathscr{O}_{X}\subset\pi_{*}\mathscr{O}_{Y}$.
Let us denote by $W$ the open subset $Y\setminus \pi^{-1}(\textrm{Sing}(X))$. Restricting the exact sequence (\ref{exact_norm})
to this open subset we get
$\pi^{*}\mathscr{F}|_{W}=\mathscr{E}_{0}|_{W}$
so $\textrm{\underline{Spec}}(S^{\bullet}(V\otimes\mathscr{E}_{0}|_{W})^{G})=\textrm{\underline{Spec}}(S^{\bullet}(V\otimes\pi^{*}\mathscr{F}|_{W})^{G})$
which means that the restriction $\pi^{*}(\tau|_{W})$ takes values in $\textrm{\underline{Spec}}(S^{\bullet}(V\otimes\mathscr{E}_{0}|_{W}))$. From the chain of immersions
\begin{equation*}
\textrm{\underline{Spec}}(S^{\bullet}(V\otimes\mathscr{E}_{0}|_{V})^{G}) \hookrightarrow \textrm{\underline{Spec}}(S^{\bullet}(V\otimes\mathscr{E}_{0})^{G}) \overset{\textrm{closed}}{\hookrightarrow} \textrm{\underline{Spec}}(S^{\bullet}(V\otimes\pi^{*}\mathscr{F})^{G})
\end{equation*}
it follows that $\pi^{*}(\tau)$ must then take values in $\textrm{\underline{Spec}}S^{\bullet}(V\otimes\mathscr{E}_{0})^{G}$, that is, the morphism $S^{\bullet}(V\otimes\pi^{*}\mathscr{F})^{G}\rightarrow\mathscr{O}_{Y}$ factorizes through the surjection
\begin{equation*}
S^{\bullet}(V\otimes\pi^{*}\mathscr{F})^{G}\rightarrow S^{\bullet}(V\otimes\mathscr{E}_{0})^{G}\rightarrow 0
\end{equation*} 
and we denote by $\tau_{0}$ the morphism of algebras
$S^{\bullet}(V\otimes\mathscr{E}_{0})^{G}\rightarrow\mathscr{O}_{Y}$. On the other hand, given a node $x\in X$, $\pi_{*}(\mathscr{E}_{0})_{x}\otimes_{\mathscr{O}_{X,x}}\mathscr{O}_{X,x}/\mathfrak{m}_{x}\simeq \mathscr{E}_{0}(y_{1})\oplus \mathscr{E}_{0}(y_{2})$. Therefore, the surjection $\pi_{*}(\mathscr{E}_{0})\rightarrow \textrm{Coker}(\beta)$ defined in Proposition \ref{canonical-bundle} induces a surjection $q^{0}_{i}:\mathscr{E}_{0}(y^{i}_{1})\oplus \mathscr{E}_{0}(y^{i}_{2})\rightarrow \textrm{Coker}(\beta)_{x_{i}}$ of dimension $e_{i}$ for each $i=1,\hdots,\nu$, which, in turn, induce a generalized parabolic structure of type $\underline{e}=(e_{1},\hdots,e_{\nu})$. From this construction, it follows that the singular principal $G$-bundle with generalized parabolic structure $(\mathscr{E}_{0},\tau_{0},\underline{q^{0}})$ of rank $r$, degree $d-\sum_{i=1}^{\nu}(r-e_{i})$ and type $\underline{e}=(e_{1},\hdots,e_{\nu})$ is a descending principal $G$-bundle and it descends to $(\mathscr{F},\tau)$.  This shows surjectivity.  On the oder hand, if $(\mathscr{E}_{1},\tau_1,\underline{q}^1)\in \Theta_{\underline{e}}^{-1}(\mathfrak{SPB}_{r,d,\underline{e}})$ is another singular principal $G$-bundle with generalized parabolic structure descending to $(\mathscr{F},\tau)$, then we have two exact sequences
$$
\xymatrix{
0\ar[r] & \mathscr{F}\ar[r]\ar@{=}[d]\ar@{..>}[rd]^{\psi} & \pi_{*}\mathscr{E}_{0}\ar[r] & R_{0}\ar[r] & 0\\
0\ar[r] & \mathscr{F}\ar[r] & \pi_{*}\mathscr{E}_{1}\ar[r] & R_{1}\ar[r] & 0\\
}
$$
Since $\mathscr{E}_{1}$ is locally free, the morphism $\psi$ induces a morphism $\iota:\mathscr{E}_{0}\rightarrow \mathscr{E}_{1}$ by adjunction, and therefore a morphism $\psi':\pi_{*}\mathscr{E}_{0}\rightarrow \pi_{*}\mathscr{E}_{1}$ making the left square commutative. This in turn implies that $\psi'$ induces a morphism $\psi'':R_{0}\rightarrow R_{1}$ making the right square commutative, and by the Short-Five  lemma, $\textrm{Ker}(\psi')=\textrm{Ker}(\psi'')$ and $\textrm{Coker}(\psi')=\textrm{Coker}(\psi'')$. However, $\textrm{Ker}(\psi'')$ must be a torsion sheaf while $\pi_{*}\mathscr{E}_{0}$ is torsion free, so we deduce that $\psi''$ is an isomorphism and, therefore, $\psi'$ is an isomorphism as well. From \cite[Huitime partie, II, Proposition 10]{Seshadri}, it follows that $\iota:\mathscr{E}_{1}\simeq \mathscr{E}_{0}$ and that this isomorphism induces an isomorphism between the parabolic structures. Now, since $\mathscr{E}_{0}\simeq \mathscr{E}_{1}$ and both,  $(\mathscr{E}_{0},\tau_{0},\underline{q}^{0})$ and $(\mathscr{E}_{1},\tau_{1},\underline{q}^{1})$, descend to $(\mathscr{F},\tau)$, we deduce that the diagram 
$$
\xymatrix{
\underline{\textrm{Spec}}(S^{\bullet}(V\otimes \mathscr{E}_{0})) \eq[rr] & & \underline{\textrm{Spec}}(S^{\bullet}(V\otimes \mathscr{E}_{1}))\\
 &Y\ar[ul]_{\tau_{0}}\ar[ur]^{\tau_{1}}&
}
$$
commutes when is restricted to $W:=Y\setminus \pi^{-1}(\textrm{Sing}(X))$. Since $\tau_{0}$ and $\tau_{1}$ are separated morphisms, we finally deduce that the diagram commutes and, therefore, $\iota:\mathscr{E}_{0}\rightarrow \mathscr{E}_{1}$ induces an isomorphism of singular principal $G$-bundles with generalized parabolic structures. This shows injectivity.
\end{enumerate}
\end{proof}

Let $r\in\mathbb{N}$, $d\in\mathbb{Z}$, $\delta\in\mathbb{Z}_{>0}$ and define $J(r):=\{\underline{e}=(e_1,\hdots,e_{\nu})\in\mathbb{N}^{\nu}| e_i\leq r\}$. For each $\underline{e}\in J(r)$ fix $\epsilon=\epsilon(\underline{e})$ and $\underline{\kappa}$ as in Proposition \ref{propimp2}. Let $\textrm{SPB}(\rho)_{r,d}^{\delta\textrm{-(s)s}}$ be the moduli space of $\delta$-(semi)stable singular principal $G$-bundles of rank $r$ and degree $d$ on the nodal curve $X$ (see \cite{AMC}).
Then Proposition \ref{propimp1} and Proposition \ref{propimp2} imply that, for each $\underline{e}\in J(r)$, there is a well defined functor
%\begin{equation}
$\textbf{D}(\rho)_{r,d(\underline{e},r),\underline{e}}^{(\underline{\kappa},\delta)\textrm{-(s)s}}\rightarrow \textbf{SPB}(\rho)_{r,d}^{\delta-(s)s}$, where $d(\underline{e},r)=d-\sum_{i=1}^{\nu}(r-e_{i})$,
%\end{equation}
and thus a proper morphism
\begin{equation}\label{finalfinal}
\Theta: \textrm{D}(\rho)_{r,d}^{(\underline{\kappa},\delta)\textrm{-(s)s}}:=  \coprod_{\underline{e}\in J(r)}\textrm{D}(\rho)_{r,d(\underline{e},r),\underline{e}}^{(\underline{\kappa},\delta)\textrm{-(s)s}}\longrightarrow \textrm{SPB}(\rho)_{r,d}^{\delta\textrm{-(s)s}}
\end{equation}
between the moduli spaces. Let $\underline{e}\in J(r)$ and let $\textrm{SPB}(\rho)_{r,d,\underline{e}}^{\delta\textrm{-(s)s}}$ be the subscheme that parametrizes singular principal $G$-bundles, $(\mathscr{F},\tau)$, with $\mathscr{F}$ a torsion free sheaf of type $\underline{e}$.
Then, by Corollary \ref{length}, $\Theta$ induces a proper morphism
$$\Theta_{\underline{e}}:   \textrm{D}(\rho)_{r,d(\underline{e},r),\underline{e}}^{(\underline{\kappa},\delta)\textrm{-(s)s}}  \longrightarrow \bigcup_{\underline{e}'\leq \underline{e}}\textrm{SPB}(\rho)_{r,d,\underline{e}'}^{\delta\textrm{-(s)s}}.$$
Let us denote by $\overline{\textrm{SPB}(\rho)_{r,d,\underline{e}}^{\delta\textrm{-s}}}$ the schematic closure in $\textrm{SPB}(\rho)_{r,d}^{\delta\textrm{-(s)s}}$, which lies in the closed subscheme $\bigcup_{\underline{e}'\leq \underline{e}}\textrm{SPB}(\rho)_{r,d,\underline{e}'}^{\delta\textrm{-(s)s}}$. Obviously $\Theta_{\underline{e}}$ maps $\overline{\Theta_{\underline{e}}^{-1}(\textrm{SPB}(\rho)_{r,d,\underline{e}}^{\delta\textrm{-s}})}$  to $\overline{\textrm{SPB}(\rho)_{r,d,\underline{e}}^{\delta\textrm{-s}}}$.

\begin{theorem}\label{equivmoduli}
If the open subscheme $\emph{D}(\rho)_{r,d(\underline{e},r),\underline{e}}^{(\underline{\kappa},\delta)\emph{-s}}\subset \emph{D}(\rho)_{r,d(\underline{e},r),\underline{e}}^{(\underline{\kappa},\delta)\emph{-ss}}$ is dense, then
$\Theta_{\underline{e}}$ induces a birational, proper and surjective morphism
$\Theta_{\underline{e}}: \emph{D}(\rho)_{r,d(\underline{e},r),\underline{e}}^{(\underline{\kappa},\delta)\emph{-ss}} \longrightarrow \overline{\emph{SPB}(\rho)_{r,d,\underline{e}}^{\delta\emph{-s}}}$.
\end{theorem}
\begin{proof}
From Proposition \ref{propimp2} and Theorem \ref{equivsets} it follows that $\Theta_{\underline{e}}$ induces an isomorphism $\Theta_{\underline{e}}^{-1}(\textrm{SPB}(\rho)_{r,d,\underline{e}}^{\delta\textrm{-s}})\simeq \textrm{SPB}(\rho)_{r,d,\underline{e}}^{\delta\textrm{-s}}$. Let us denote by $\mathcal{W}_{\underline{e}}$ the dense open subscheme of $\textrm{D}(\rho)_{r,d(\underline{e},r),\underline{e}}^{(\underline{\kappa},\delta)\textrm{-(s)s}}$ parametrizing descending principal bundles with generalized parabolic structure such that $\textrm{dim}(C_{1}^{i})+\textrm{dim}(C_{2}^{i})=2(r-e_{i})$ for $i=1,\hdots,\nu$ (see Remark \ref{funtorial-des}). From Proposition \ref{propimp1} and Remark \ref{funtorial-des2} it follows that $\Theta_{\underline{e}}^{-1}(\textrm{SPB}(\rho)_{r,d,\underline{e}}^{\delta\textrm{-s}})=\mathcal{W}_{\underline{e}}\cap \textrm{D}(\rho)_{r,d(\underline{e},r),\underline{e}}^{(\underline{\kappa},\delta)\textrm{-s}}$. Therefore, it is a dense open subscheme. Finally, Since $\Theta_{\underline{e}}$ is proper, the isomorphism $\Theta_{\underline{e}}^{-1}(\textrm{SPB}(\rho)_{r,d,\underline{e}}^{\delta\textrm{-s}})\simeq \textrm{SPB}(\rho)_{r,d,\underline{e}}^{\delta\textrm{-s}}$ extends to a surjective and proper morphism $\Theta_{\underline{e}}: \textrm{D}(\rho)_{r,d(\underline{e},r),\underline{e}}^{(\underline{\kappa},\delta)\textrm{-(s)s}} \longrightarrow \overline{\textrm{SPB}(\rho)_{r,d,\underline{e}}^{\delta\textrm{-s}}}$.
\end{proof}

\end{document}